\documentclass[letterpaper,10pt,reqno]{amsart}
\usepackage[usenames,dvipsnames]{xcolor}

\usepackage[portrait,margin=3.5cm]{geometry}
\usepackage{mathrsfs}

\usepackage[foot]{amsaddr}
\usepackage{amssymb,amsthm,amsmath,amsfonts,amsbsy,latexsym,dsfont}
\usepackage[english]{babel}

\usepackage{color}
\usepackage{graphicx}
\usepackage{bbm}
\usepackage{comment}
\usepackage{microtype}
\definecolor{MyDarkblue}{rgb}{0,0.08,0.50}
\definecolor{Brickred}{rgb}{0.65,0.08,0}

\usepackage[colorlinks,citecolor=blue,urlcolor=blue,linkcolor=blue]{hyperref}
\usepackage{todonotes}
\hypersetup{
colorlinks=true,       
    linkcolor=MyDarkblue,          
    citecolor=Brickred,        
    filecolor=red,      
    urlcolor=cyan           
}

\newtheorem*{theorem*}{Theorem}
\newtheorem{theorem}{Theorem}[section]
\newtheorem{lemma}[theorem]{Lemma}

\newtheorem{proposition}[theorem]{Proposition}
\newtheorem{corollary}[theorem]{Corollary}
\newtheorem{problem}[theorem]{Problem}
\newtheorem{definition}[theorem]{Definition}
\newtheorem{assumption}[theorem]{Assumption}
\newtheorem{remark}[theorem]{Remark}
\newtheorem{claim}[theorem]{Claim}

\newcommand{\Pv}{\mathbb{P}}

\newcommand{\Ev}{\mathbb{E}}

\newcommand{\cH}{\mathcal{H}}

\newcommand{\bE}{\mathbb{E}}

\newcommand{\bN}{\mathbb{N}}
\newcommand{\bP}{\mathbb{P}}\newcommand{\bR}{\mathbb{R}}

\newcommand{\sss}{\scriptscriptstyle}

\newcommand{\CE}{{\mathcal{E}}}

\newcommand*{\CMD}{{\mathrm{CM}}_n(\boldsymbol{d})}
\newcommand*{\CMDL}{{\mathrm{CM}}_n(\boldsymbol{d}, \boldsymbol{L})}
\newcommand*{\CMDP}{\mathrm{CM}_n^{p(d)}(\boldsymbol{d})}

\newcommand*{\ECMDL}{{\mathrm{ECM}}_n(\boldsymbol{d}, \boldsymbol{L})}

\newcommand{\e}{{\mathrm e}}

\numberwithin{equation}{section}

\newcommand{\R}{\mathbb{R}}
\newcommand{\N}{\mathbb{N}}
\newcommand{\Z}{\mathbb{Z}}

\newcommand{\equalsd}{\overset{d}{=}}

\newcommand{\CU}{\mathcal {U}}
\newcommand{\CV}{\mathcal {V}}

\newcommand{\CH}{\mathcal {H}}

\newcommand*{\ve}{\varepsilon}

\newcommand*{\al}{\alpha}

\newcommand*{\fl}[1]{\lfloor{#1}\rfloor}

\newcommand*{\be}{\begin{equation}}
\newcommand*{\ee}{\end{equation}}
\newcommand*{\ba}{\begin{aligned}}
\newcommand*{\ea}{\end{aligned}}
\newcommand*{\barr}{\begin{array}{c}}
\newcommand*{\earr}{\end{array}}
\def \toinp    {\buildrel {\Pv}\over{\longrightarrow}}
\def \toindis  {\buildrel {d}\over{\longrightarrow}}

\newcommand*{\wit}{\widetilde}

\newcommand*{\ind}{\mathbbm{1}}
\def\namedlabel#1#2{\begingroup
    #2%
    \def\@currentlabel{#2}%
    \phantomsection\label{#1}\endgroup
}

\newcommand{\dl}{d_L(u,v)}
\newcommand{\dle}{d_L^e(u,v)}
\newcommand{\he}{\mathcal{H}_n}
\newcommand{\FL}{F_L(x)}

\newcommand{\an}{\sum_{i=1}^{\left\lfloor \frac{\log\log n}{|\log(\tau-2)|} \right\rfloor} F_L^{\sss{(-1)}}\left(\mathrm{e}^{-\left(\frac{1}{\tau-2}\right)^i}\right)}

\newcommand{\Xk}{X_{k}^{\sss{(n)}}}

\begin{document}
	\title[Weighted distances in the configuration model]{
Weighted distances in scale-free configuration models}

	\date{\today}

	\author[Adriaans]{Erwin Adriaans}
	\address{Department of Mathematics and
	    Computer Science, Eindhoven University of Technology, P.O.\ Box 513,
	    5600 MB Eindhoven, The Netherlands.}
\author[Komj\'athy]{J\'ulia Komj\'athy}
	\email{e.l.a.adriaans@student.tue.nl, j.komjathy@tue.nl}

\begin{abstract}
In this paper we study first-passage percolation in the configuration model with empirical degree distribution that follows a power-law with exponent $\tau \in (2,3)$. We assign independent and identically distributed (i.i.d.)\ weights to the edges of the graph.  We investigate the weighted distance (the length of the shortest weighted path) between two uniformly chosen vertices, called typical distances.  When the underlying age-dependent branching process approximating the local neighborhoods of vertices is found to produce infinitely many individuals in finite time -- called explosive branching process -- Baroni, Hofstad and the second author showed in \cite{BarHofKomdist} that typical distances converge in distribution to a bounded random variable. The order of magnitude of typical distances remained open for the $\tau\in (2,3)$ case when the underlying branching process is not explosive.
We close this gap by determining the first order of magnitude of typical distances in this regime for arbitrary, not necessary continuous edge-weight distributions that produce a non-explosive age-dependent branching process with infinite mean power-law offspring distributions. 
 This sequence tends to infinity with the amount of vertices, and, by choosing an appropriate weight distribution,  can be tuned to be any growing function that is $O(\log\log n)$, where $n$ is the number of vertices in the graph.  
We show that the result remains valid for the the erased configuration model as well, where we delete loops and any second and further edges between two vertices.
\end{abstract}
\maketitle
\section{Introduction}
Every logistic company wants to be the fastest, cheapest and deliver on time. In order to achieve this, the routes they are driving should be (near-) optimal, meaning they should be the least costly and fastest for them in order to be competitive. This is just an example where weighted distances in a network play an important role. Other examples include the spreading of epidemics through society, the spreading of rumours, videos and advertisement through (online) social network, and several other processes spreading on the internet.  

The recent interest in understanding complex networks and processes on these networks motivates the study of more and more elaborate models for these (weighted) networks. The analysis of processes on these models often reveal finer topological aspects of the models themselves.  And, vice versa, the organisation and topology of a network affect the behaviour of different processes on the network. Many real-life networks turn out to share some common properties, one of them being  that the degree distribution follows a power-law  \cite{Falo1999, Newm01}, examples include the world-wide web \cite{BarAlbJeo00}, the movie-actor collaboration network \cite{Bara1999}, the network of citations of scientific publications \cite{Red98}, and many more. Another common property is the small-world phenomenon, popularized by Millgram \cite{Milg67} as: ``everyone on this planet is separated from anyone else by only six people''. Mathematically speaking, a network exhibits the small-world property if the minimal amount of connections to go from one node to another is of order $\log(n)$ or $\log\log(n)$ for ultra- small worlds, with $n$ the amount of nodes in the network. This effect is not only seen in social networks, but also in neurological networks like the brain \cite{Acha2006, BulSpo09} or food webs \cite{MonSol02}. A third common property is clustering as pointed out by Watts and Strogatz \cite{WatStr98}. High clustering means that two vertices in the graph are more likely to be connected to one another when they have a common neighbor. This is a common feature in e.g.\ social networks.


The natural way to model a network from a mathematical point of view is to see this as a graph, 
where nodes are represented by vertices and their connections by edges. Since real-life networks are large, models often involve randomness to determine the presence of edges between the vertices. 
Random graph models that incorporate the (first two) above mentioned properties often serve as null-models for the analysis of real-life networks. Examples include variation of inhomogeneous random graphs such as 
the Chung-Lu or Norros-Reitu model \cite{ChuLu01, Reittu:2004}, the configuration model \cite{BenCan78, Boll80}, and the preferential attachment model \cite{AlbBar02}. Spatial variants are introduced to incorporate clustering, e.g. hyperbolic random graphs \cite{BogPapKri10}, geometric inhomogeneous random graphs \cite{BriKeuLen15}, scale-free percolation \cite{DeiHof13}, spatial preferred attachment \cite{AieBon08, JacMor13}, etc.  

When modeling the spread of information in a network, edge weights to the edges can be added that represent the passage time of the information through the edge.  The weighted distance is then the weight of the path with smallest total weight, corresponding to the passage time of the information from one vertex to the other.  When the edge-weights are i.id., the study of the resulting weighted graph is often called \emph{first-passage percolation} (FPP). Introduced by Hammersley and Welsh \cite{HamWel65} for the grid $\Z^d$, FPP can be seen as a flow, starting from a vertex, flowing through the edges at a rate equal to the respective edge-weights, the weighted distance corresponding to the time it takes the front of the flow to reach the other vertex. 

First passage percolation has been studied on the Erd{\H o}s-R\'enyi random graph see \cite{BHH11}, on inhomogeneous random graphs see \cite{KolKom15}.  FPP on the configuration model with finite mean degrees for exponential edge-weights is treated in \cite{BHH10}, with finite variance degrees (i.e., power law exponent at least $3$) and arbitrary edge-weight distributions in \cite{bhamidi2017}, and for infinite variance degrees (power-law exponent $\in (2,3)$) for a class of edge-weights \cite{BarHofKomdist}.  In particular, \cite{BarHofKomdist} determines the weighted distance when the edge-weights fall into what they call the \emph{explosive} class. In this case, weighted distances converge in distribution (see Theorem \ref{thm:explosive} below), which heuristically means that regardless of how large the size of the network gets, the average weighted distance in the networks \emph{stays bounded}. This explains the observed phenomena of extremely fast information spread in e.g.\ online networks such as meme spreading or viral spreading. 
The other class, where the weight distribution is `non-explosive' is further studied in \cite{BarHofKom16}, for the special case  when the edge weights are of the form $1+X$. For this case \cite{BarHofKom16} shows that the weighted distance is tight around the typical graph distance (that is, $2\log \log n/|\log (\tau-2)|$, where $\tau$ is the power-law exponent), if and only if the extra weight $X$ falls into the explosive class.  

 In this paper we  investigate the missing case, i.e., FPP on the configuration model with infinite variance degrees (power-law exponent $\in(2,3)$ and i.i.d.\  edge-weights that fall into the `non-explosive' class. We determine the first order of weighted distance in the highest generality, thus, together with \cite{BarHofKom16}  providing an (almost) full picture of weighted distances in the $\tau \in (2,3)$ case. We also extend our results to the \emph{erased configuration model}, when only one edge of every multiple edge is kept. 

\emph{Structure.} 
In the next section we introduce the configuration model and state our results, as well as discuss related results and open problems.  In Section \ref{s:couple} we develop a coupling to branching processes (BPs), and state and prove some ingredient lemmas about the degrees and weighted distances within these BPs. In Section \ref{s:perc} we develop a crucial tool to prove the upper bound of the main result, degree-dependent percolation.  In Section \ref{s:proof}, we prove the main result and extend it to the erased configuration model. 

\emph{Notation.} We say that a sequence of events $(\CE_n)_{n \in \bN}$ holds with high probability (whp) if $\lim_{n \to \infty} \bP(\CE_n) = 1$. For a sequence of random variables $(X_n)_{n \geq 1}$,  we say than $X_n$ converges in probability to a random variable X,  shortly $X_n  \toinp X$,  if for all $\ve > 0, \lim_{n \to \infty} \bP(|X_n-X| > \ve ) = 0$. Similarly, we say that $X_n$ converges in distribution to a random variable X,  shortly $X_n  \toindis X$,  if  $\lim_{n \to \infty}\bP(X_n \le x) \to \bP(X \le x)  $ for all $x\in \R$ where $\bP(X \le x) $ is continuous.
 For a non-decreasing right-continuous function $F(x)$ the generalised inverse of $F$ is defined as $F^{(-1)}(x) := \inf\{y \in \bR : F(y) \geq x \}$. For an edge $e=(x,y)$ we write $L_e$ for the associated edge-length on $e$.
We write lhs and rhs for the left-hand side and right-hand side, respectively.

\section{Model \& Results} 
In this section we introduce the weighted configuration model and present our results. Then we discuss related research and describe some open problems. 
\subsection{The model}
\label{sec:model} We consider the configuration model $\CMD$ on $n$ vertices with degree sequence $\boldsymbol{d} = \{d_1, \ldots ,d_n\}$. Let $\he := \sum_{v \in [n]} d_v$, the sum of the degrees with $[n]$ := $\{1, 2, \ldots, n\}$. If $\CH_n$ is odd we add an additional half-edge to vertex $n$, this does not further influence the analysis and we retain from discussing this issue further.
Given the degree sequence, the model is constructed as follows: To every vertex $v \in [n]$ we assign $d_v$ half-edges, then we take a uniform random matching  of the half-edges, where any two matched half-edges form an edge of the graph. The resulting random graph is denoted by $\CMD$. After constructing the edges, we assign each edge $e$ an i.i.d.\ edge-length $L_e$ from distribution $L$. We denote the resulting weighted random graph by $\CMDL$.
We assume that the empirical distribution function of the degrees, defined as $F_n(x):=\frac1n \sum_{v\in [n]} \ind_{\{d_v\le x\}}$, satisfies the conditions for a power-law distribution, as given in the following assumption.
\begin{assumption}[Power-law tail behavior]\label{ass:degree-dist} There exist $\tau \in (2,3)$, $\gamma \in (0,1)$, $C > 0$ and $\al>1/2$ such that for all  $x \in [0, n^{\al})$,
	\be \label{eq:Fn} \frac{1}{x^{\tau-1}}\mathrm{e}^{-C(\log x)^{\gamma}} \leq 1 - F_n(x) \leq \frac{1}{x^{\tau-1}}\mathrm{e}^{C(\log x)^{\gamma}}, \ee 
Additionally, we assume that $\min_{v \in [n]} d_v \geq 2$. 
\end{assumption}
Under Assumptions \ref{ass:degree-dist}, \cite{JanLuc09}, there is a giant component of size $n(1-o(1))$, thus two uniformly chosen vertices lie whp in the same connected component.
Let $D_n$  denote a random variable with distribution function $F_n$, the degree of an uniformly chosen vertex in $[n]$. We define $B_n$ as the (size biased version of $D_n$)-1.
\be \label {eq:size-biased} \bP(B_n = k) := \frac{k+1}{\CH_n} \sum_{v\in [n]} \ind_{\{d_v = k+1\}} = \frac{k+1}{\bE[D_n]}\bP(D_n = k+1). \ee
We write $F_{B_n}$ for the distribution function of $B_n$. As shown in \cite{HofKom16}, $F_{B_n}$ also satisfies a similar bound as \eqref{eq:Fn}, namely, for some $C^\star>0$,
\be \label{eq:FB-bounds} \frac{1}{x^{\tau-2}}\mathrm{e}^{-C^\star(\log x)^{\gamma}} \leq 1 - F_{B_n}(x) \leq \frac{1}{x^{\tau-2}}\mathrm{e}^{C^\star(\log x)^{\gamma}}. \ee
To be able to relate models with different values of $n$ to each other, we pose an additional assumption.
\begin{assumption}[Limiting distributions]\label{ass:tv}
	There exist distribution functions $F_D(x), F_{B}(x)$ such that for some $\kappa>0$,
	\[ \max\{ \mathrm{d}_{\sss{\mathrm{TV}}}(F_n, F), \mathrm{d}_{\sss{\mathrm{TV}}}(F_{B_n}, F_{B}) \} \le n^{-\kappa},\]
	where $\mathrm{d}_{\sss{\mathrm{TV}}}(F, G):=\tfrac12 \sum_{x\in \N} |F(x+1)-F(x) - (G(x+1)-G(x)) |$ is the total variation distance between two (discrete) probability measures.
\end{assumption}
We denote the random variables following the distribution $F_D$ and $F_B$ of Assumption \ref{ass:tv} by $D$ and $B$. Clearly Assumption \ref{ass:tv} implies $D_n \toindis D$ and $B_n \toindis B$. Since $F_D$ and $F_B$ are independent of $n$, it is elementary to show that they satisfy \eqref{eq:Fn} and \eqref{eq:FB-bounds} for all $x\in \N$. 

The goal of this paper to study the weighted distances in $\CMDL$, that now we define.
\begin{definition}[Graph- and Weighted distance, Hopcount]\label{def:weighted-dist}
	Let $u$ and $v$ be two vertices in $\CMDL$. Then the \emph{graph distance} $d_G(u,v)$ is the number of edges used by the shortest path between $u$ and $v$. The \emph{weighted distance}  or \emph{$L$-distance} between $u$ and $v$ is defined as
	\be \label{eq:weighted-dist} \dl := \min_{\pi: u \rightarrow v} \sum_{e \in \pi} L_e, \ee
	where the minimum is taken over all paths connecting $u$ to $v$ present in $\CMDL$. We set $\dl = 0$ if $u = v$ and $\dl = \infty$ if $u$ and $v$ are not connected. We define $d_H(u,v)$, the \emph{hopcount}, as the number of edges on the optimal path realising $\dl$. Finally, for two sets of vertices $A, B$, $\mathrm d_L(A,B):=\min_{x\in A, y \in B} \mathrm d_L(x,y)$.
\end{definition}

The following theorem states the main result of this paper:
\begin{theorem}[Weighted distances]
	\label{thm:mainthm}
	Consider the weighted configuration model $\CMDL$ satisfying Assumptions \ref{ass:degree-dist}-\ref{ass:tv} and let $u$ and $v$ be two uniformly chosen vertices from $[n]$. Suppose that the  distribution function $\FL$ of $L $ satisfies  that
	\be\label{eq:weigthdist} \sum_{k = 1}^{\infty}F_L^{\sss{(-1)}}(\mathrm{e}^{-\mathrm{e}^k}) = \infty \ee
	Then, for the weighted distance,
	\be\label{eq:mainthm} \dl \; \Big / \; 2 \an \toinp 1. \ee
	For the hopcount, for all $\ve>0$
	\be \lim_{n\to \infty}\Pv\Big(d_H(u,v) \Big / 2\frac{\log \log n}{|\log (\tau-2)|} \ge 1-\ve\Big)=0, \ee
	and whp, there exist at least one path of length at most $1+\ve$ times the denominator in \eqref{eq:mainthm}, with number of edges at most $(1+\ve)2 \log \log n/|\log (\tau-2)|$.
\end{theorem}

Convergence in distribution of the hopcount around $2\log \log n/|\log (\tau-2)|$ remains an open question, since the upper bound does not follow from our techniques. Namely, we cannot exclude the possibility of a much longer path with optimal total edge-length.  

The weighted erased configuration model is defined as follows.  After $\CMDL$ is constructed, we remove all self-loops and, if there are multiple-edges between two vertices, one of the edges is chosen uniformly at random independent of the edge weights and the other edges are deleted. The resulting graph is called the weighted erased configuration model, shortly $\ECMDL$. Let us denote the L-distance in this graph by $\dle$ and the hopcount by $d_H^e(u,v)$.

\begin{theorem}[Weighted distances in the erased configuration model]
	\label{thm:erasedthm}
	Consider the erased configuration model $\ECMDL$ satisfying Assumptions \ref{ass:degree-dist}-\ref{ass:tv} and let $u$ and $v$ be uniformly chosen from $[n]$. Suppose that the distribution function $\FL$ of $L$ satisfies 	\eqref{eq:weigthdist}. Then the results of Theorem \ref{thm:mainthm} remain valid for $\dle$ and $d_H^e(u,v)$ as well. 
\end{theorem}

\begin{remark}[I.i.d. degrees]\normalfont
	Using concentration techniques it can be shown that Assumptions \ref{ass:degree-dist} and \ref{ass:tv} are satisfied whp when the degrees are i.i.d.\ coming from a background distribution function $F(x)$ satisfying \eqref{eq:Fn} for all $x\in \N$, see \cite{bhamidi2017}. 
	\end{remark}
	
\begin{remark}\normalfont[Explosive vs non-explosive edge-weight distributions]
Given a particular distribution $L$ for the edge weights, convergence vs divergence of the sum in \eqref{eq:weigthdist} is elementary to check, and it depends on the behaviour of $F_L$ around $0$. The steeper $F_L$ at the origin, the smaller the sum in \eqref{eq:weigthdist}: the sum converges  e.g.\ if $F_L$ increases as a polynomial in around $0$,  which is the case for exponential, uniform, Gamma  distributions. Distributions with support separated away from $0$ always give a divergent sum, and distributions with inverse $F_L^{(-1)}(z)=O(1/\log\log (1/z))$ also diverge. This corresponds to the family of distributions $F_L(t)=\exp\{-C\exp\{-c/t^\beta\}\}$, that give explosion for $\beta<1$ but non-explosion for $\beta\ge1$.
\end{remark}

Note that \eqref{eq:weigthdist} does not require that $F_L$ is continuous. 	
By setting the edge weights to be deterministic and equal to 1 in Theorem \ref{thm:mainthm}, we obtain the following corollary. Stronger results about the graph distance were already obtained in \cite{HHZ07, HofKom16}.

\begin{corollary}[Graph distances]
	Consider the configuration model $\CMD$  satisfying Assumptions \ref{ass:degree-dist}-\ref{ass:tv} and let $u$ and $v$ be uniformly chosen vertices from $[n]$. Then
	\be d_G(u,v) \Big/ 2\left\lfloor\frac{\log \log n}{|\log(\tau-2)|}\right\rfloor \toinp 1. \ee
\end{corollary}

A counterpart of Theorem \ref{thm:mainthm} is \cite[Theorem]{BarHofKomdist}, which we cite here for comparison.
\begin{theorem}[Weighted distances with explosive edge-weights \cite{BarHofKomdist}]\label{thm:explosive}
Consider the weighted configuration model $\CMDL$ satisfying Assumptions \ref{ass:degree-dist}-\ref{ass:tv} and let $u$ and $v$ be two uniformly chosen vertices from $[n]$. Suppose that the  distribution function $\FL$ of $L $ satisfies  that the sum in 
\eqref{eq:weigthdist} converges. Then	
\be\label{eq:mainthm2} \dl  \toindis Y^{\sss{(1)}}+Y^{\sss{(2)}}, \ee
	where $Y^{(1)}, Y^{(2)}$ are i.i.d. copies of some a.s.\ finite random variable. 

\end{theorem}
Theorems \ref{thm:mainthm} and \ref{thm:explosive} together describe typical distances in the configuration model with power law degrees, with exponent $\tau\in(2,3)$ for all edge-weight distributions $L$. Next we discuss some related literature and pose some open problems.

\subsection{Discussion and open problems}
\emph{Relation to age-dependent branching processes.}
The configuration model has a tree-like local structure. Since most cycles are long, the local neighborhood of a uniformly chosen vertex  exploration around a vertex can be coupled to a branching process (BP). When the edge-weights are incorporated in the model and in the coupling, this BP becomes  age-dependent.  In an age-dependent BP, individuals have an i.i.d.\ lifetime and give birth to their i.i.d. number of  offspring upon death. Let us denote such a BP with offspring distribution $X$ and life-time distribution $\sigma$ by $\mathrm{BP}(X, \sigma)$. Let us write $\mathrm{BP(D, X, \sigma)}$ for $D$ i.i.d. copies of $\mathrm{BP}(X, \sigma)$. Then, the local neighborhood of a vertex in $\CMDL$ can be approximated by $\mathrm{BP}(D,F_{B}, L)$. Explosion of a BP means that the BP produces infinitely many individuals in finite time, with positive probability. In 2013 Amini et al \cite{AmiDevGriOlv13} gave a necessary and sufficient condition for for the explosion of $\mathrm{BP}(X,L)$, for offspring distributions $X$ that satisfy $\Pv(X\ge x) \ge x^{-1-\ve}$ for some $\ve>0$. In an unpublished note \cite{Kom16}, under the stronger assumption that $X$ satisfies $ x^{-1-\ve} \le \Pv(X\ge x) \le x^{-\ve}$ for some $\ve>0$,  the second author simplified this criterion to the sum in \eqref{eq:weigthdist} being finite. The criterion \eqref{eq:weigthdist} comes from the following observation: the $L$-length of any path in a BP leading to infinity can be lower bounded by the sum of the minimum edge-lengths in each generation. 
In generation $k$, the number of individuals is  double exponential in $k$. The minimum of this many i.i.d.\ random variables of distribution $L$
is approximately the $k$th term in the sum in \eqref{eq:weigthdist}. If this sum is infinite, the BP cannot explode, thus the summability of the minimums in each generation is necessary for the BP to explode.  In \cite{AmiDevGriOlv13}, the authors showed that this notion of minimum-summability is sufficient as well by constructing an algorithm that finds an infinite path with finite total length. 

To show distributional  convergence of weighted distances in $\CMDL$, as in Theorem \ref{thm:explosive}, when the underlying BP explodes was the content of \cite{BarHofKomdist}. 
It remained open to characterise the growth of weighted distances when explosion does not happen. It follows from  \cite{AmiDevGriOlv13} that in the nonexplosive case, for offspring distribution $X$ satisfying \eqref{eq:FB-bounds}, the time to reach the first individual in generation $\ell$ grows as 
\[ \sum_{k=1}^\ell F_L^{{-1}} ( \exp\{ - (\tau-2)^{-k}\}).\] This, combined with the fact that the graph distance of a typical vertex to a maximal degree vertex is $\log\log n/|\log (\tau-2)|$, gives a strong intuitive explanation for the formula for $\dl$ in Theorem \ref{thm:mainthm}. 
Unfortunately, the BP approximation for $\CMD$ fails much earlier than reaching the maximal degree vertex, and the BP techniques do not reveal enough information on the structure of the optimal path leading generation $k$ of the BP, in particular, they do not provide good enough lower bounds on the degrees along the path. These are the reasons why we need to use a different technique, degree-dependent percolation, to show the upper bound on $\dl$.
Unfortunately, this technique is not fine enough to show distributional convergence of the fluctuations of $\dl$ around its typical value. 
 
 \begin{problem}[Tightness, distributional covergence]
	Consider $\CMDL$ satisfying Assumptions \ref{ass:degree-dist}-\ref{ass:tv} and let $u$ and $v$ be uniformly chosen vertices from $[n]$. Suppose the distribution function $\FL$ satisfies \eqref{eq:weigthdist}. Determine the conditions under which
		\be \dl \; - \; 2 \an, \qquad   d_H(u,v) - 2 \frac{\log\log n}{|\log(\tau-2)|} \ee  
	are tight sequences of random variables. Do these sequences converge in distribution?
	\end{problem}

Infinite mean degrees, i.e., when $\tau \in (1,2)$, is investigated in \cite{BhaHofHoo10,EskHofHooZna05}, where the authors show that the graph distance is whp $2$ or $3$, the weighted distance converges to the sum of two random variables.
Finite variance degrees, $\tau > 3$, is studied in \cite{BHH10,bhamidi2017,HofHooVan05a,EskHofHoo08}. In this case typical graph distances are of order $\log n$, weighted distances scale as a constant times $\log n$ with converging fluctuations around this value, while the hopcount, centered around another constant times $\log n$, satisfies a central limit theorem.

It still remains open to characterise weighted distances for the boundary exponents, i.e., when $\tau\in \{2, 3\}$. For the $\tau=3$ case, even the explosion of the underlying age dependent BP is an open question. For $\tau=2$, local neighborhoods grow faster than double-exponential and the precise growth depends sensitively on the slowly varying function involved, thus the techniques used here do not apply directly.  
  \begin{problem}[$\tau=2$ or $3$]
	Characterise weighted distances for the case when the degree distribution follows a power law (with a  slowly varying function correction term) when $\tau=2$ and when $\tau=3$. 
	\end{problem}
We further expect that similar results hold for a large class of power-law graph models, specially in the $\tau\in(2,3)$ regime, including inhomogeneous random graphs (e.g. the Chung-Lu or Norros-Reitu models), spatial models such as the geometric inhomogeneous random graphs and scale-free-percolation.

\subsection{Overview of the proof}\label{s:overview}
Next we give an overview of the  proof of Theorem \ref{thm:mainthm}. The proof consists of two parts, a lower and an upper bound that use slightly different techniques.  
\subsubsection{Lower bound} 
Let us denote the graph distance ball of radius $k$ around a vertex $q$ in $\CMD$ by $B^G_k(u)$, and the set of vertices precisely at graph distance $k$ away from $q$ by $\Delta B^G_k(u)$. 
For the proof of the lower bound we show that $B^G_k(u), B^G_k(v)$ can be coupled to two independent branching processes (BPs), with the first generation having distribution function $F_D$ and all further generations having distribution function $F_B$ from Assumption \ref{ass:tv}. We show that the coupling can be maintained until two random indices $\kappa_n(u), \kappa_n(v)$ such that polynomially many vertices in $n$ are found around both vertices $u,v$, and that $B^G_{\kappa_n(u)}(u), B^G_{\kappa_n(v)}(v)$ are whp disjoint. 
Since any path connecting $u,v$ must intersect the boundaries of these sets, we obtain the lower bound
\be\ba\label{eq:lower-heuristic}d_L(u,v) &\ge d_L(u, \Delta B^G_{\kappa_n(u)}(u)) + d_L(v, \Delta B^G_{\kappa_n(v)}(v))\\
&\ge \sum_{q=u,v}\sum_{i=0}^{\kappa_n(q)-1} \min_{x\in \Delta B^G_i(q), y \in \Delta B^G_{i+1}(q) }\{L_{(x,y)}\} \ea\ee
where we obtained the second line by lower bounding $d_L(q, \Delta B^G_{\kappa_n(q)}(q))$ by the sum of the minimal edge lengths connecting $\Delta B^G_i(q)$ to $\Delta B^G_{i+1}(q)$ over $i$. We show that this sum of minima is larger than $(1-\ve)$ times the denominator of the lhs of \eqref{eq:mainthm} whp. 

\subsubsection{Upper bound} \label{s:proof-overview-upper}
The upper bound also couples the neighborhoods $B_k(u), B_k(v)$ to two disjoint BPs, but we exploit the coupling only until we reach vertices $u_{K_n}, v_{K_n}$ of degree at least $\widetilde K_n$, for some carefully chosen $\wit K_n$ that tends to infinity with $n$, but $d_L(u, u_{K_n})$ and $d_L(v, v_{K_n})$ are still of negligible length compared to the denominator of the lhs of \eqref{eq:mainthm}. Then we connect the vertices $u_{K_n}$ and $v_{K_n}$ using degree-dependent percolation that we describe now. 

The idea for degree-dependent percolation originates from \cite{BarHofKom16}, and it is an extension of a construction by Janson \cite{Jan09}. In the percolated graph, we keep edges independently of each other, with probabilities that depend on the degrees of the end vertices of the edge. We use the i.i.d.\ edge lengths to realise the percolation, i.e., an edge $e=(x,y)$ is kept if and only if its edge length satisfies
\be\label{eq:threshold} L_{x,y} \le \xi(d_x, d_y),\ee
for some appropriately chosen threshold function $\xi(\cdot,\cdot)$. We use a result from \cite{BarHofKom16, Jan09},
that states that the percolated graph can be looked at as a subgraph $G^r$ of a configuration model with a new degree sequence $\mathbf{d^r}$.  We choose $\xi$ in such a way that the new degree sequence still satisfies the power-law condition in \eqref{ass:degree-dist}, for the same $\tau$ but possibly different $C,\gamma$. Note that $G^r$ is a subgraph of the original $\CMDL$, and as a result any path present in $G^r$ was necessarily also present in $\CMDL$. We show that $u_{K_n}, v_{K_n}$ has percolated degree at least $K_n$.

Then, we construct two paths, emanating from $u_{K_n}$ and $v_{K_n}$, and reaching vertices $\wit u, \wit v$ of percolated degree at least $n^{(\tau-2)/(\tau-1)}$, respectively, in this percolated graph. We control the (growing) degrees of vertices along these paths and as a result \eqref{eq:threshold} gives an upper bound on the edge-lengths along these paths.  
More precisely, analogous to \cite{BarHofKom16}, we define a sequence $y_i(K_n)$ with $y_0 = K_n$ and layers in the graphs $\Gamma_i := \{v \in [n]: d_v \geq y_i(K_n)\}$, for $0\leq i \leq i_{\max}$ with $i_{\max}$ the number of layers. We show that a vertex in $\Gamma_i$ is connected to a vertex in $\Gamma_{i+1}$ whp, moreover the total error probability over all the layers tends to zero as $K_n\to \infty$. Thus whp there exist paths from $u_{K_n}, v_{K_n}$ where the $i$th vertex along the path has degree at least $y_i(K_n)$.  Finally, we connect the vertices $\wit u$, $\wit v$ in $G^r$ via a path of length at most four using vertices with degree at least $n^{1/2}$. The length of the constructed path is at most
\be\label{eq:upper-heuristic} \dl \leq d_L(u, u_{K_n}) + d_L(v, v_{K_n}) + 2\sum_{i = 0}^{i_{\max}} \xi(y_{i}(K_n),y_{i+1}(K_n)) + 3 \xi(n^{\al(\tau-2)},n^{1/2}). \ee
The first two terms on the rhs, coming from the branching processes, are negligible due to the choice of $K_n$, and the last term also since it tends to zero with $n$. With the proper choice of $\xi(\cdot, \cdot)$, the middle term becomes at most $(1+\ve/2)$ times the denominator of \eqref{eq:mainthm}.
\section{Exploration around two vertices}\label{s:couple}
The goal of this section is couple the neighborhoods $B_k(u), B_k(v)$ of the two uniformly chosen vertices $u,v$ to two independent BPs. We first show that, for $q\in\{u,v\}$ the coupling can be maintained until $k=\kappa_n(q) = \log\log n / |\log(\tau-2)|$+ a tight random variable. Then, using the growth of the BPs, we make \eqref{eq:lower-heuristic} quantitative by giving a whp lower bound on the minimum of edge-lengths connecting consecutive generations in the BPs.

As a preparation for the upper bound, as in \eqref{eq:upper-heuristic}, we determine $M_n$, the number of generations needed to reach a vertex with degree at least $K_n$, that we denote by $q_{k_n}$, for $q\in\{u,v\}$. Finally, we give an upper bound on $d_L(q, q_{K_n})$ for $q\in\{u,v\}$. 

\subsection{Coupling of the exploration to a branching process}
First we explain the coupling of the neighborhoods of the vertices $u$ and $v$ to branching processes. The coupling uses an exploration, where we reveal the pairs of half-edges and thus the neighbors of vertices together with their degrees one-by-one, in a breadth-first-search manner.  By $\CU_t, \CV_t$ we denote the subgraphs consisting of vertices at graph distance of at most $t$ from $u$ and $v$, respectively. The forward degree of a vertex $v$ in the exploration denotes then the number of new (not previously discovered) neighbors of a vertex upon exploration.
We slightly adjust \cite[Lemma 2.2]{HofKom16}) to our setting, since Assumptions \ref{ass:degree-dist}-\ref{ass:tv} are a special case of the assumptions of \cite[Lemma 2.2]{HofKom16}).  An alternative formulation and proof can be found in \cite[Proposition 4.7]{BHH10}. 
\begin{lemma}[Coupling error of the exploration process, \cite{HofKom16}]
	\label{lem:coupling-err}
	Consider  $\CMD$ satisfying Assumptions \ref{ass:degree-dist}-\ref{ass:tv}. Then, in the exploration process started from two uniformly chosen vertices $u$ and $v$, the forward degrees $(\Xk)_{k\leq s_{n}}$ of the first $s_n$ newly discovered vertices can be coupled to an i.i.d.\ sequence $B_k$ from distribution $B$ as in Assumption \ref{ass:tv}. So, there is a coupling  $(X_k^{\sss{(n)}}, B_k)_{k\le s_n}$ with the following error bound
	\begin{align} \nonumber
	 \Pv( \exists k \le s_n,   \Xk \neq B_k) \le & \ Cs_n^{(2\tau-2-2\ve)/(\tau-\ve)}n^{(2-\tau+\ve)/(\tau-\ve)} \\&+ C s_n^2 n^{\frac{(1+\ve)}{\tau-1}-1} + s_n n^{-\kappa}. \label{eq:couple-error-tv}
	\end{align}
\end{lemma}
An immediate corollary is the following:
  \begin{corollary}[Whp coupling of the exploration to BPs, \cite{HofKom16}] \label{corr:couple}
	In the configuration model satisfying Assumption \ref{ass:degree-dist} and \ref{ass:tv}, let $t$ be such that
	\be |\CU_t\cup \CV_t| \leq 2\min\{n^{(1-(1+\ve)/(\tau-1)-\delta)/2},n^{-(\tau-2-2\ve)/2(\tau-1-\ve)}, n^{(2-\tau+\ve)/(\tau-\ve)}, n^{\kappa-\delta} \}=:2n^{\theta(\delta)}\ee 
	for some $\delta>0$.  Then $(\CU_t, \CV_t)$ can be whp coupled to two i.i.d.\ BPs  with generation sizes $(Z_{k}^{\sss{(u)}}, Z_{k}^{\sss{(v)}})_{k>0}$  with distribution $F_B$ for the offspring in the second and further generations, and with distribution $F$ for the offspring in the first generation.
\end{corollary}
The proofs of Lemma \ref{lem:coupling-err} and Corollary \ref{corr:couple} can be found in \cite[Section 2]{HofKom16}.

Davies in \cite{D78} shows that for a BP with offspring distribution satisfying the tail behavior in \eqref{eq:FB-bounds}, the sequence of random variables $Y_k := (\tau - 2)^n\log(Z_k)$  converges almost surely. It is elementary to extend his result to a BP where the root has a different offspring distribution (see \cite{BarHofKomdist} for details). Having this result in mind for the two BPs coupled to the neighborhoods of $u,v$, 
we rewrite the generation sizes as
\be\label{eq:BP-size} Z_k^{\sss{(q)}} =: \exp{\Big \{ } \left(\frac{1}{\tau - 2}\right)^kY_k^{\sss{(q)}}{\Big \} }. \ee
Fixing a small $\delta>0$, we define for $q\in\{u,v\}$,
\be\label{def:kappa_n} \kappa_n(q):=\max\{k: Z_k^{\sss{(q)}} \le n^{\theta(\delta)} \},\ee
and then  Corollary \ref{corr:couple} implies that $\Delta B_{\kappa_n(q)}(q)$ has size $Z_{\kappa_n(q)}^{\sss{(q)}}$ since the coupling can still be maintained.
Combining \eqref{eq:BP-size} and \eqref{def:kappa_n} we obtain that $Y_{\kappa_n(q)}^{\sss{(q)}}$ converges \emph{in distribution} to two independent copies of the same random variable $Y$. The convergence now is only distributional, since there is no coupling between the BPs for different values of $n$. Using \eqref{eq:BP-size} and \eqref{def:kappa_n}, we can provide an \emph{implicit} description of $\kappa_n(q), q\in\{u,v\}$. 

\begin{claim}[Last generation of the exploration]
	\label{lem:lastgenexpl}
	Consider  $\CMD$  satisfying assumptions \ref{ass:degree-dist}-\ref{ass:tv}. Let $u$ and $v$ be two uniformly chosen vertices. Then we can couple the BFS-exploration around $u$ and $v$ to two BPs until generation that has the implicit representation
	\be\label{eq:kappa_n} \kappa_n (q)= \frac{\log\log n + \log (\theta(\delta)f_n(q)/Y_{\kappa_n(q)}^{\sss{(q)}})}{|\log(\tau-2)|}\text{  for } q\in\{u,v\} , \ee
	where $\kappa_n(q)$ is an integer, where $Y_{\kappa_n(q)}^{\sss{(q)}}$, for $q\in \{u,v\}$ are independent and converge in distribution, and $f_n(q)\in(\tau-2,1]$ describes the exponent $\theta(n)f_n(q)$ that satisfies $Z_{\kappa_n(q)}^{\sss{(q)}}=n^{\theta(n)f_n(q)}$.
\end{claim}
The message of this claim is that the coupling can be maintained until $\log \log n/|\log(\tau-2)|$ + a tight random variable many generations.
\begin{proof}
Fix $\delta>0$ small enough and set  $\theta(\delta)$ as in Corollary \ref{corr:couple}. Corollary \ref{corr:couple} then implies that the coupling error converges to zero as long as $|B_{k_1}(u)\cup B_{k_2}(v)|\le 2n^{\theta(\delta)}$. T definition of $\kappa(q)$ in \eqref{eq:kappa_n} implies that this is indeed satisfied for $k_1:=\kappa_n(u), k_2:=\kappa_n(v)$. 
	The value of $\kappa_n(q)$ in \eqref{eq:kappa_n} is  obtained by an elementary rearrangement of the formula 
	\eqref{eq:BP-size} when $k$ is replaced by $\kappa_n(q)$, and we took the integer part of the obtained expression. Note that $\kappa_n(q)$ is well-defined this way since the generation sizes are increasing double-exponentially for all large enough $k$ due to \eqref{eq:BP-size} and the fact that $Y_k^{\sss{(q)}}$ would converge if we would let $k$ tend to infinity, and as a result the total size of $B_k(q)$ is $1+o(1)$ times the last generation size. 
\end{proof}
Next we make \eqref{eq:lower-heuristic} quantitative by  giving a lower bound on the length of the path from $q$ to generation $\kappa_n(q)$. Recall that $\Delta B_k(w)$ is the set of vertices at distance $k$ from a vertex $w$ in $\CMD$. 
\begin{lemma}[Lower bound on shortest path length]
	\label{lem:lowerboundexpl}
	Consider $\CMD$  satisfying assumptions \ref{ass:degree-dist}-\ref{ass:tv} with i.i.d.\ edge lengths from distribution $L$ with distribution function $F_L$ satisfying \eqref{eq:weigthdist}. Let $u,v$ be two uniformly chosen vertices. Then, for $q\in \{u,v\}$, with $\kappa_n(q)$ as in \eqref{eq:kappa_n}, 
	\be\label{eq:gen-lower} \lim_{n \to \infty} \bP\left(d_L(u, \Delta B_{\kappa_n(q)}(q)) > (1-\ve) \an \right) = 1.\ee
\end{lemma}
\begin{proof} Under the coupling between the neighborhoods of $u,v$ to the BPs established in Corollary \ref{corr:couple} and Claim \ref{lem:lastgenexpl}, for all $i\le \kappa_n(q),  \Delta B_{i}(q)=Z_i^{\sss{(q)}}$, the size of generation $i$ in the BP coupled to the neighborhood of $u$. Using the idea in \eqref{eq:lower-heuristic},
	$d_L(u, \Delta B_{\kappa_n(q)}(q))$ is longer than the sum of the minimum edge-lengths between consecutive generations. I.e.,
	\be \label{eq:wdkappa-lowerb} d_L(u, \Delta B_{\kappa_n(q)}(u)) \geq \sum_{i = 0}^{\kappa_n(q)-1} \min \{L_{i, 1}^{\sss{(q)}}, \ldots L_{i, Z_i^{\sss{(u)}}}^{\sss{(q)}}\}, \ee
	where $L_{i,j}^{\sss{(u)}}$ are i.i.d.\ for all $i,j$ and $q\in\{u,v\}$. 
	We let 
	\be \label{eq:Cn} C_n := \max\{\sup_{1\leq k\leq\kappa_n(q)}Y_k^{\sss{(q)}}, h(n)\} \ee
	 with $h(n)$ a function defined later on. By \eqref{eq:BP-size}, $Z_i^{\sss{(q)}} = \exp\{ \left(\tau - 2\right)^{-i}Y_i^{\sss{(q)}}\}$, thus, using, \eqref{eq:Cn}, 
	 \be \label{eq:ziq-upper}Z_i^{\sss{(q)}} \le \exp\{ \left(\tau - 2\right)^{-i}C_n \}.\ee
	 For i.i.d.\ $L_j$, the following tail bound holds for any $N\in \N, z(N)>0$:
	\be\label{eq:lower-min} \bP(\min_{j\leq N} L_j > z(N)) = (1-F_L(z(N)))^N \geq 1 - NF_L(z(N))\ee which is at most $1/N^\xi$ when we set, for some $\xi > 0, z(N) := F_L^{\sss{(-1)}}(1/N^{1+\xi})$.
Using that the minimum in \eqref{eq:wdkappa-lowerb} is non-increasing when increasing the number of variables involved, by \eqref{eq:ziq-upper}, we can set $N$ to be $\exp\{ \left(\tau - 2\right)^{-i}C_n \}$ to estimate the $i$th term in \eqref{eq:wdkappa-lowerb} from below using \eqref{eq:lower-min}. Conditioning on the value $C_n, \kappa_n(q)$, combined with a union bound, yields that the inequality
		 \begin{align} \label{eq:min-lowerbound} \sum_{i = 0}^{\kappa_n(q)-1} \min \{L_{i, 1}, \ldots L_{i, Z_i^{\sss{(u)}}}\} \geq \sum_{i = 0}^{\kappa_n(q)-1  } F_L^{(-1)}\left(e^{- (\tau - 2)^{-i}C_n(1+\xi)}\right), \end{align} holds 
	 with error probability (conditioned on $C_n, \kappa_n(q)$) at most 
	\be \label{eq:errprob} E(C_n) :=  \sum_{i = 0}^{\infty} e^{-\left(\frac{1}{\tau - 2}\right)^iC_n\xi}  \leq C_1 e^{-C_n\xi} \ee
	for some constant $C_1 > 0 $. Combining \eqref{eq:wdkappa-lowerb} and \eqref{eq:min-lowerbound} yields that under the coupling, with error probability given in \eqref{eq:errprob}, 
	\be\label{eq:ineq-wl} d_L(q, \Delta B_{\kappa_n(q)}(q)) \geq \sum_{i = 0}^{\kappa_n(q)-1} F_L^{(-1)}\left(e^{-\left(\frac{1}{\tau - 2}\right)^iC_n(1+\xi)}\right). \ee
	Next we transform the rhs to match the format in \eqref{eq:gen-lower}.
 With $\lceil a\rceil=\min\{m\in \Z, m\ge a\}, \rfloor b\rfloor = \max\{m\in \Z, m\le a\}$, we use the following  inequalities valid for monotonous non-increasing functions $g$ with $g(0)<\infty$:
		\be \label{eq:tranformationineq} \sum_{k = \lceil a\rceil +1}^{\lfloor b \rfloor} g(k)\  {\buildrel (\star) \over \leq} \int_a^b g(x) \mathrm dx, \qquad \ \int_{0}^{z+1} g(x) \mathrm dx\  {\buildrel (\triangle) \over \leq} 
	\ \sum_{k = 0}^z g(k). \ee
	We use $(\triangle)$ to bound the rhs of \eqref{eq:ineq-wl} from below, then we carry out the variable transformation $1/(\tau -2)^{x} C_n(1+\xi) = 1/(\tau -2)^{y}$, and transform the integral back to a sum using $(\star)$. The variable transformation shifts the summation boundaries by $\widetilde{C}_n := \log (C_n(1+\xi)) / |\log(\tau - 2)|$, and we obtain that 
	\be\label{eq:sum-lowerb} \sum_{i = 0}^{\kappa_n(q)-1} F_L^{(-1)}\left(\mathrm{e}^{-\left(\frac{1}{\tau - 2}\right)^iC_n(1+\xi)}\right) \geq \sum_{i = \lceil \widetilde{C}_n + 1 \rceil}^{\kappa_n(q) + \lfloor \widetilde{C}_n \rfloor} F_L^{(-1)}\left(\mathrm{e}^{-\left(\frac{1}{\tau - 2}\right)^i}\right). \ee We bound the upper summation boundary on the rhs from below.
	Recall $C_n$ from \eqref{eq:Cn}, then
	 	\be \label{eq:ineqCn} \lfloor \widetilde{C_n} \rfloor = \left\lfloor \frac{\log ( (1+\xi)\max\{\sup_{1\leq k\leq\kappa_n}Y_k^{\sss{(q)}}, h(n)\} )}{|\log(\tau - 2)|}  \right\rfloor  \geq \frac{\log Y_{\kappa_n(q)}}{|\log(\tau-2)|}-1\ee
Using now the formula for $\kappa_n(q)$ from \eqref{eq:kappa_n}, 
\be\label{eq:upper-boundary} \kappa_n(q) + \lfloor \widetilde{C}_n \rfloor \ge \frac{\log \log n + \log (\theta(n)f_n(q)(\tau-2))}{|\log (\tau-2)|}.\ee
	Next, the lower summation boundary on the rhs of \eqref{eq:sum-lowerb} is not $1$, and, if $\wit C_n\to \infty$, then this might cause too much difference from the desired sum in \eqref{eq:gen-lower}.  Thus,	 for any fixed $\ve > 0$ we define
	\be\label{eq:Rn} R_n(\ve) := \max_z \left\{ \sum_{k=1}^{z-1}F_L^{(-1)}\left(e^{-\left(\frac{1}{\tau - 2}\right)^i}\right) \leq \frac{\ve}{2} \sum_{i = 1}^{\lfloor\log\log n/|\log (\tau-2)|\rfloor } F_L^{(-1)}\left(e^{-\left(\frac{1}{\tau - 2}\right)^i}\right) \right\} .\ee
	Since the sum on the rhs between the brackets tends to infinity with $n$, so will $R_n(\ve)$. Setting $\widetilde{C_n} = R_n(\ve)$, combined with \eqref{eq:upper-boundary} and the fact that the summands tend to zero then implies that 
	\be \label{eq:almost-final-lowerb} \sum_{i = \lceil \tilde{C}_n + 2 \rceil}^{\kappa_n + \lfloor \widetilde{C}_n \rfloor } F_L^{(-1)}\left(e^{-\left(\frac{1}{\tau - 2}\right)^i}\right) > (1-\ve)\an, \ee
	as desired. 
The choice $\widetilde{C}_n = \log (C_n(1+\xi)) / |\log(\tau - 2)|=R_n(\ve)$, establishing the choice	$h(n) = 1/(1+\xi)\left(\tau-2\right)^{-R_n(\ve)}$ in \eqref{eq:Cn}.  Since $R_n(\ve)$ tends to infinity, so will $C_n$, ensuring that the error probability in \eqref{eq:errprob} tends to zero as well. This finishes the proof of the lower bound.
\end{proof}
Next we do some preparations for the proof of the upper bound. First we investigate the number of generations we need to explore to reach a vertex of degree at least $\widetilde{K_n}$. 
\begin{lemma}[Generations needed to reach degree $\widetilde{K_n}$]
	\label{lem:generation-Kn}
	Consider $\CMD$  satisfying Assumptions \ref{ass:degree-dist}-\ref{ass:tv}. Let $u,v$ be two uniformly chosen vertices and $(\widetilde{K_n})_{n\geq 1} = O(\log n )$ a sequence that tends to infinity with $n$. Then, for $q\in \{u,v\}$, for any $M$ with $ M|\log(\tau-2)|>1$, 
	\be \lim_{n \to \infty}\bP\left( \max_{w \in \Delta B_{M \log \log \wit K_n}(q)} d_w < \widetilde{K_n} \right) = 0. \ee

	\end{lemma}
\begin{proof} For brevity we write $M_n:=M \log \log \wit K_n$. By Corollary \eqref{corr:couple}, and Claim \ref{lem:lastgenexpl}, $B_k(q), k\le \kappa_n(q)$, for $q\in\{u,v\}$ can be coupled to two BPs, where, in each generation the degrees are i.i.d. from distribution $B$ as in Assumption \ref{ass:tv}. Recall \eqref{eq:BP-size},  write $1+\delta = M|\log(\tau-2)|$ and condition on whether $Y_{M_n}^{\sss{(q)}}$ is less than $1/(\log\widetilde{K_n})^{\delta/2}$ or not. Then
	\begin{align}
	\nonumber \; \bP\left( \max_{v \in \Delta B_{M_n}(u)} d_v \leq \widetilde{K_n}\right) \leq  & \; \bP \left( \max_{v \in Z_{M_n}^{\sss{(u)}}} B_v \leq \widetilde{K_n} \mid Y_{M_n}^{\sss{(u)}} > 1/(\log \widetilde{K_n})^{\delta/2}\right) \\ & + \bP \left(Y_{M_n}^{\sss{(u)}} < 1/(\log \widetilde{K_n})^{\delta/2}\right). \label{eq:maxdegreeineq}
	\end{align}
	By Davies \cite{D78}, the limiting variable $\lim_{n\to \infty} Y_k^{\sss{(u)}}$ is almost surely positive on survival of the BP. 
 By Assumption \ref{ass:degree-dist}, $\Pv(B\ge 1)=1$ and thus the BP cannot go extinct and therefore $\bP(Y_{M_n}^{\sss{(u)}} = 0) = 0$. Thus, the second term on the rhs in \eqref{eq:maxdegreeineq} converges to zero. For the first term,
	\be \nonumber\bP\Big( \max_{v \in Z_{M_n}^{\sss{(u)}}} B_v \leq \widetilde{K_n} \mid Z_{M_n}^{\sss{(u)}}\Big) = (F_B(\widetilde{K_n}))^{Z_{M_n}^{\sss{(u)}}}. \ee
	Using the lower bound on $F_B$ from \eqref{eq:FB-bounds}, with $L(x) := \exp\{-c(\log x)^\gamma\}$  we obtain
	\be \label{eq:maxdegreebound} (F_B(\widetilde{K_n}))^{Z_{M_n}^{\sss{(u)}}} \leq \text{{\Big(}}1-\frac{L(\widetilde{K_n})}{\widetilde{K_n}^{\tau-2}}\text{{\Big)}}^{Z_{M_n}^{\sss{(u)}}} \leq \exp\{-L(\widetilde{K_n})Z_{M_n}^{\sss{(u)}}/\widetilde{K_n}^{\tau-2}\} .\ee
	Using that $Z_{M_n}^{\sss{(u)}} = \exp\{(\tau-2)^{-M_n}Y^{\sss{(u)}}_{M_n}\}$ by \eqref{eq:BP-size} in \eqref{eq:maxdegreebound}, we obtain that
	\begin{align}
		& \nonumber\bP \Big( \max_{v \in Z_{M_n}^{\sss{(u)}}} B_v \leq \widetilde{K_n} \mid Y_{M_n}^{\sss{(u)}} > \frac{1}{(\log \widetilde{K_n})^{\delta/2}}\Big)
		\leq \exp \text{\Big\{}-\frac{\exp\{-L(\log \widetilde{K_n})^\gamma\}\exp\{(\log \widetilde{K_n})^{1+\delta/2}\}}{\widetilde{K_n}^{\tau -2}}  \text{\Big\}}  \\
		& \ \ \ \nonumber \leq \exp \Big \{ -\exp \left\{ (\log\widetilde{K_n})^{1+\delta/2}(1-L(\log\widetilde{K_n})^{\gamma - 1-\delta/2})-(\tau-2)(\log\widetilde{K_n})^{-\delta/2} \right\} \Big \} \\
		& \ \ \ \label{eq:probfirstpart} \leq \exp \Big\{ -\exp \big\{ 1/2(\log\widetilde{K_n})^{1+\delta/2} \big\} \Big\} \xrightarrow{n \rightarrow \infty} 0.
	\end{align}
	where we used that $M_n = M\log \log K_n$, with $1+\delta = M|\log (\tau - 2)|$, the last inequality holds for $\widetilde{K_n}$ large. Thus both probabilities on the rhs in \eqref{eq:maxdegreeineq} tend to zero as $\widetilde{K_n}$ tends to infinity, which completes the proof.
\end{proof}

In the proof of the upper bound of the main theorem we run the exploration algorithm until we reach a vertex of degree $\widetilde{K_n}$. The path to this vertex is the sum of i.i.d.\ copies of edge weights. We show that whp this sum is less than some $\ve_1 > 0$ times the denominator of the lhs of \eqref{eq:mainthm}.
\begin{lemma}[Upper bound on the length of the path in the exploration process]
	\label{lem:upperboundexpl}
	Let $(L_i)_{i\ge 1}$ be i.i.d.\  from distribution $F_L$ satisfying \eqref{eq:weigthdist}. Then for all $\ve_1 >0$ there exists a choice of $M_n$ such that $M_n$ tends to infinity with $n$ and
	\be \label{eq:pathlength} \lim_{n \to \infty} \bP\Bigg(\sum_{i = 0}^{M_n}L_i < \ve_1\an\Bigg) = 1. \ee
	
\end{lemma}
The lemma immediately follows from the following, more general result.
\begin{claim} Let $(L_i)_{i\ge 1}$ be i.i.d. random variables from distribution $F_L$, and let $(a_m)_{m\ge 1}$ be an arbitrary sequence that tends to infinity as $m\to \infty$. Then, there exists a deterministic sequence $(z_L(m))_{m\ge 1}$ with $\lim_{m\to \infty} z_L(m)=\infty$ such that
\[\lim_{m\to \infty}\Pv\Big( \sum_{i=1}^{z_L(m)} L_i \le a_m\Big) =1.\]
  
\end{claim}
\begin{proof}
 We distinguish two cases, based on the tail behavior of $L$. When $F_L$ does not satisfy any of these cases, $L$ can be stochastically dominated by a random variable that does satisfy (at least) one of these cases and then the result follows by a simple stochastic domination argument.  
 \emph{Case (1)}: $\Ev[L]<\infty$. In this case, for some $\delta \in (0,1)$, let
		\be \label{eq:Mn1} z_L(m) := \max_{z\in \bN} \{ z \leq a_m^\delta \} \ee	
	Clearly  $z_L(m)$ tends to infinity with $m$ when $a_m$ does.  Markov's inequality implies that 	
	\be\label{eq:Markov-111} \bP\left(\sum_{i = 0}^{z_L(m)}L_i \geq   a_m\right) \leq \frac{z_L(m) \bE[L]}{ a_m} \to 0. \ee 
\noindent\emph{Case (2):}	 $\Ev[L]=\infty$ and additionally $\bP(L_i > x) \leq 1/g(x)$ for some  non-decreasing  function $g(x)$. 
For some small $\delta, \ve_2>0$, we  define $z_L(m)$ implicitly by
	\be \label{eq:Mn2}  z_L(m)g^{(-1)}(\left(z_L(m)\right)^{1+\ve_2}) =  a_m^{1-\delta}, \ee
	where $g^{(-1)}(x)=\inf\{y\in \R: g(y)\ge x\}$.
	 Since $g(x)$ is non-decreasing and $a_m$ tends to infinity, $z_L(m)$ tends to infinity as well. Note that when $g(x)\ge x^a$ for some $a\in(0,1)$, capturing regularly varying cases, a lower bound on \eqref{eq:Mn2} can be explicitly calculated: 
	 \[ z_L(m)\ge a_m^{(1+(1+\ve_2)/a)^{-1}(1-\delta)}
	 \]
	 To estimate the lhs of \eqref{eq:Markov-111} in this case, we use a truncation argument.
	 We condition \eqref{eq:pathlength} on the maximum of the $L_i$ being larger than $T_m:= g^{\sss{(-1)}}\left(\left(z_L(m)\right)^{1+\ve_2}\right)$ or not, which gives us the following upper bound
	\be \label{eq:Mn2cond} \bP\left(\sum_{i = 0}^{z_L(m)}L_i \geq a_m\right) \leq \bP\left(\exists i \leq z_L(m): L_i > T_m\right) + \bP\left(\sum_{i = 0}^{z_L(m)}L_i\ind_{\{L_i \leq T_m\}}  \geq a_m\right). \ee
	First we focus on the first term in \eqref{eq:Mn2cond}. Using that $\Pv(L>x)=1/g(x)$ and the value $T_m$,
	\be \label{eq:Mn2first} \bP\left(\exists i \leq z_L(m): L_i > T_m\right) \leq z_L(m) \bP\left(L_i \geq T_m \right) \leq  \frac{z_L(m)}{(z_L(m))^{1+\ve_2}} = \left(z_L(m)\right)^{- \ve_2}, \ee
	that tends to zero as $m$ tends to infinity. Next we investigate the second term in \eqref{eq:Mn2cond} which we bound with Markov's inequality,
	\be \label{eq:Mn2markov} \bP\left(\sum_{i = 0}^{z_L(m)}L_i\ind_{\{L_i \leq T_m\}}  \geq a_m\right) \leq  \frac{z_L(m)\bE [L_i\ind_{\{L_i \leq T_m\}}] }{a_m} \ee
Now we  observe that $\bE [L_i \ind_{\{L_i \leq T_m\}}]\le T_m$,
	and use this bound on the rhs of \eqref{eq:Mn2markov}, and \eqref{eq:Mn2},
	\be
	 \label{eq:Mn2secondcase} \frac{z_L(m)\bE [L_i\ind_{\{L_i \leq T_m\}}] }{a_m} \leq \frac{z_L(m)T_m}{a_m} \le \frac{a_m^{1-\delta}}{a_m} \le a_m^{-\delta}
	\ee
Combining \eqref{eq:Mn2first} and \eqref{eq:Mn2secondcase} implies that \eqref{eq:Mn2cond} tends to zero as $m$ tends to infinity. This finishes the proof. \end{proof}

\section{Degree-dependent percolation on the configuration model}\label{s:perc}
In this section we make the degree dependent-percolation precise, that we have described in Section \ref{s:proof-overview-upper}.  Percolation for the configuration model was studied in \cite{Jan09} and later adjusted for the degree-dependent version in \cite{BarHofKom16}. 
We start by giving the definition of an induced subgraph.
\begin{definition}[Induced subgraph]
	Let $S$ be a set of vertices. The induced subgraph of G on vertex set S is the largest subgraph of G with edges that have both endpoints in S. We denote the induced graph of a graph $G$ restricted to the vertices in a set $S$ by $G_{|S}$.
\end{definition}
 Let $ p(d):\bN \rightarrow [0,1]$
be a monotone decreasing function of $d$. For a half-edge $s$ we write the percolation probability shortly as $p_s := p(d_{v(s)})$ with $v(s)$ the vertex that $s$ is attached to and $d_{v(s)}$ the degree of vertex $v(s)$. Now we define two different ways to percolate the configuration model, after that we show equality in distribution for the two different percolated graphs.

\begin{definition}[Edge percolation] \label{def:perc2}
	Consider a configuration model $\CMD$ with half-edges already paired into edges. Delete any edge between vertices with degrees $d, d'$ in the graph independently of all other edges with probability $p(d) p(d')$. We denote the resulting graph by $\widetilde{\mathrm{CM}}_n^{p(d)}(\boldsymbol{d})$.
\end{definition}
As described in Section \ref{s:proof-overview-upper}, we can realize the egde percolation on $\CMDL$ by using the i.i.d.\ edge-lengths in $(L_e)_e$ as auxiliary variables to determine which edge to keep. Then, the threshold function $\xi(d, d')$ as in \eqref{eq:threshold} must satisfy $\Pv(L\le \xi(d, d'))=p(d)p(d')$ for all $d,d'\in \N$.

\begin{definition}[Half-edge percolation] \label{def:perc1}
	Given a degree sequence $\boldsymbol{d} = (d_1, \ldots ,d_n)$ and a half-edge $s$, we keep a half-edge with probability $p_s$ independently. If we do not keep it, then we create a new vertex with one half-edge corresponding to the deleted half-edge. We call the newly created vertex and half-edge artificial. We denote the total number of artificial vertices by $A$. After this procedure is carried out for all $s \in [\cH_n]$ we pair all the half-edges uniformly at random, (including the artificial ones as well). At last we take the induced subgraph on the $n$ original vertices. We denote the resulting graph by $\CMDP$.\end{definition}
	By denoting the number of half-edges that are kept at vertex $i$ by $d_i^r$, and $1(A)$ a sequence with $A$ repetitions of the value $1$, $\CMDP$ is nothing but $CM_{n+A}(d^r, 1(A))_{|[n]}$, i.e., the induced subgraph of the first $n$ vertices of a configuration model with $n+A$ vertices, and degree sequence that is $d_i^r$ for $i\le n$ and $1$ for $i\ge n$.

A result in \cite{BarHofKom16} is the following: 

\begin{corollary}[Equality in distribution of two percolated graphs]
	\label{cor:equalitydist}
	Consider a function $p(d)$, the degree-dependent percolation $\widetilde{\mathrm{CM}}_n^{p(d)}(\boldsymbol{d})$ as in Definition \ref{def:perc2}.  Then $\CMDP \equalsd \widetilde{\mathrm{CM}}_n^{p(d)}(\boldsymbol{d})$, where  $\CMDP$ is the half-edge percolation as described in Definition \ref{def:perc1}. \end{corollary}
	The message of Corollary \ref{cor:equalitydist} is that we can understand the (connectivity) properties of the graph after the degree-dependent edge percolation $\widetilde{\mathrm{CM}}_n^{p(d)}(\boldsymbol{d})$ by studying a configuration model $\mathrm{CM}_{n+A}(d^r, 1(A))$ restricted to the first $n$ vertices. In some sense this corollary enables to change the order of percolation and pairing. So,  now on we focus on studying the properties of $\CMDP=\mathrm{CM}_{n+A}(d^r, 1(A))|_{[n]}$.
Importantly,  we need to control the new degree sequence in $\CMDP$. Recall that the vector $\boldsymbol{d^r} := \{d_1^r, \ldots ,d_n^r\}$ denotes the number of kept half edges attached to vertices in $[n]$ in Defintion \ref{def:perc1}.  Let us write $F_n^r(x):=\frac1n\sum_{i=1}^n\ind_{\{d_i^r\le x\}}$. The goal is to find conditions on $p(d)$ such that $F_n^r(x)$ still satisfies the conditions of \eqref{eq:Fn}, when $F_n$ did so. 

\begin{lemma}[Empirical degree distribution after percolation]
	\label{lem:percdegrees}
	Consider $\CMD$ with degree sequence satisfying Assumption \ref{ass:degree-dist}. Perform half-edge percolation as described in Definition $\ref{def:perc1}$ on $\CMD$ with percolation function $p(d)$ satisfying 
	\be \label{eq:perccondition} p(d) > b\exp\{-c(\log(d))^\eta\} \ee
	for some constants $b, c > 0$ and $\eta \in (0,1)$. Then there exists a $\theta$ such that for all $x \in [\theta, n^{\al}]$ the empirical degree distribution $F_n^r(x)$ of the degrees after percolation still satisfies Assumption \ref{ass:degree-dist}, except the condition on the minimal degree, with the same $\tau, \al$, but possibly different $\gamma\in (0,1)$.\end{lemma}

\begin{proof}
By Definition \ref{def:perc1}, half-edges are kept independently, and thus, given $d_i$,  $d_i^r\  {\buildrel d\over=}\  \mathrm{Bin}(d_i, p(d_i))$, where $\mathrm{Bin}(n, p)$ is a binomial random variable with parameters $n$ and $p$. As a result the random variables $(d_i^r)_{i\leq n}$ are independent given the initial degrees $(d_1, \ldots, d_n)$. The upper bound in Assumption \ref{ass:degree-dist} for $F_n^r$ is elementary since 
	\be 1 - F_n^r(x) = \frac{1}{n}\sum_{i=1}^n \ind_{\left\{\mathrm{Bin}(d_i, p(d_i)) > x\right\}} \leq \frac{1}{n}\sum_{i=1}^n \ind_{\left\{d_i > x\right\}} = 1 - F_n(x). \ee
	 Next we show the lower bound. First we define for all $x < n^\alpha$ 
	\be S(x) := \left\{v : d_v \geq s(x) \right\} \ee
	where $s(x)>x$ is a function of $x$ that is defined later. Clearly, 
	\be 1 - F_n^r(x) \geq \frac{1}{n}\sum_{i \in S(x)} \ind_{\left\{\mathrm{Bin}(d_i, p(d_i)) > x\right\}} \ee
	We choose the value of $s(x)$ such that the probability that the indicators within the sum are 1 with high enough probability, for all $i\in S(x)$. Namely, if  we choose $y(x)$ such that the expectation of the binomial, $d_ip(d_i)$, is higher than $2x$ for all vertices in $S(x)$, then we can use the concentration of binomial random variables \cite[Theorem 2.21]{H10} to get an upper bound on the probability that the indicator functions are 1. Let 
	\be \label{eq:y(x)} s(x) = \frac{2x}{b}\mathrm{e}^{2c (\log(2x/b))^\eta} \ee then we find that
	\begin{align*}
	s(x)p(s(x)) & \geq 2x\mathrm{e}^{2c(\log (2x/b))^\eta}\mathrm{e}^{-c(\log(2x/b) + {2c(\log (2x/b))^\eta}))^\eta} \\
	&= 2x\mathrm{e}^{2c(\log (2x/b))^\eta}\mathrm{e}^{-c\left(\log(2x/b) (1 + 2c(\log (2x/b))^{\eta-1})\right)^\eta}.
	\end{align*}
	Since $\eta<1$, the factor $1 + 2c(\log (2x/b))^{\eta-1}$ in the exponent of the last factor is at most $3/2$ whenever  $x \geq \frac{b}{2}\exp\left\{(4c)^{1-\eta}\right\} := \frac{b}{2}\widehat{\theta}$. Using this fact and  the monotonicity of $dp(d)$, we find that for all $d_i > s(x)$,
	\be d_ip(d_i) \geq s(x)p(s(x)) \geq 2x e^{c/2(\log 2x/b)^\eta} \geq 2x. \ee
	Then, for all $d_i > s(x)$, by \cite[Theorem 2.21]{H10},
	\be\ba\label{eq:bin-conc} \bP\left(\mathrm{Bin}(d_i, p(d_i)) > x \right) &\leq \bP\left(\mathrm{Bin}(d_i, p(d_i) > \frac{s(x)p(s(x))}{2} \right) \\
	&\leq e^{-s(x)p(s(x))/8} \leq e^{-x/4} < \frac{1}{8}, \ea\ee
	whenever $x > 4\log8$. Using this we get for all $x \geq \max\left\{\frac{b}{2}\widehat{\theta}, 4\log8\right\} := \theta$
	\begin{align*}
	&\bP \left( n(1-F_n^r(x))\leq \frac{S(x)}{4}\right) \leq 
	\bP \left( \sum_{i \in S(x)} \ind_{\left\{\mathrm{Bin}(d_i, p(d_i)) > x\right\}} \leq  \frac{S(x)}{4} \right) \\
	\;\;\;&\leq \bP\left(\text{Bin}(|S(x)|, 7/8) \leq \frac{|S(x)|}{4}\right) \leq \mathrm{e}^{-|S(x)|/8}.
	\end{align*}
	Combining this estimate with a union bound,
	\be \label{eq:setunion} \bP\left(\exists x \in [\theta, n^\alpha]: n(1-F_n^r(x)) \leq \frac{|S(x)|}{4}\right) \leq \sum_{x=\theta}^{n^\alpha} e^{-|S(x)|/8} \leq n^\alpha e^{-|S(n^\alpha)|/8}, \ee
	since $|S(x)|$ decreases as $x$ increases. Using \eqref{eq:Fn} $|S(x)|$ can be bounded from below as follows
	\begin{equation*}
	 	\label{eq:setsize} |S(x)| = n(1-F_n(s(x))) \geq n \frac{1}{s(x)^{\tau-1}}\mathrm{e}^{-c(\log s(x))^\eta}.
	\end{equation*}
	It is elementary to calculate that, with  $s(x)$ as in \eqref{eq:y(x)}, the rhs satisfies satisfies the lower bound in Assumption \ref{ass:degree-dist}, with the same $\tau, \al$, while the new value of $\gamma$ is $\max\{\gamma^{\text{old}}, \eta\}$.
	Using this bound for $|S(x)|$ within the probability sign in \eqref{eq:setunion} and for $|S(n^\alpha)|$ on the rhs of \eqref{eq:setunion} we arrive at:
	\be \bP\left(\exists x \in [\theta, n^\alpha]: 1-F_n^r(x) \leq \frac{1}{s(x)^{\tau-1}}\mathrm{e}^{-c(\log s(x))^\eta}\right) \leq \mathrm{e}^{\alpha \log n} \mathrm{e}^{-n^\ve \mathrm{e}^{-c(\log n)^\eta}} \xrightarrow{n \rightarrow \infty} 0. \ee
	
\end{proof}
Next we prepare more for the proof of the upper bound of Theorem \ref{thm:mainthm}, by comparing  the degree of a fixed vertex before and after the half-edge percolation. This will be used to ensure that $u_{K_n}, v_{K_n}$ in Section \ref{s:proof-overview-upper} still has sufficiently high degree in the percolated subgraph. 
\begin{lemma}[Degree after percolation vs original degree]
	\label{lem:degreeafterperc}
 Apply  half-edge percolation as described in Definition \ref{def:perc1} with percolation function $p(d)$  satisfying \eqref{eq:perccondition} on $\CMD$. Let  $\widetilde{K_n} = O(\log n)$ an arbitrary sequence that tends to infinity with $n$. We define
	\be \label{eq:Kn} K_n := \sup \left\{m : 2m \leq \widetilde{K_n}b\mathrm{e}^{-c(\log \widetilde{K_n})^\eta} \right\}. \ee
	 Then a vertex $w$ with $d_w\ge \widetilde{K_n}$ in $\CMD$ has degree at least $K_n$ in $\CMDP$ whp.
\end{lemma}
\begin{proof}
	First we investigate the expected degree of a vertex after percolation. Consider a vertex $w$ with degree $d_w$, as before $d^r_w$ denotes the degree after the half-edge percolation. Recall that $d^r_w
\ {\buildrel d \over =} \ \mathrm{Bin}(d_w, p(d_w))$. Thus 
	\be \nonumber \label{eq:expectedperddegree} \bE[d_w^r] \geq \bE[ \mathrm{Bin}(d_w, b\mathrm{e}^{-c(\log d_w)^\eta}) ] = d_w b\mathrm{e}^{-c(\log d_w)^\eta} = b\mathrm{e}^{\log d_w(1 -c(\log d_w )^{\eta-1})}. \ee
	The rhs is monotone increasing in $d_w$, and tends to infinity as $d_w \to \infty$. Therefore, by setting  $K_n$ as in \eqref{eq:Kn}, $K_n$ tends to infinity when $\widetilde{K_n}$ does.  By \eqref{eq:Kn},  the expectation of a $\mathrm{Bin}(d_w, p(d_w))$, for any $d_w\ge \wit K_n$, is larger than $2K_n$. Knowing that, we can use the concentration of binomial random variables \cite[Theorem 2.21]{H102} to obtain a bound on the probability that the binomial is smaller than $K_n$, i.e.
	\be \nonumber	\bP\left( \mathrm{Bin}(d_w, p(d_w)) < K_n| d_w\ge \wit K_n\right) \leq \exp \left\{-K_n/4\right\}, \ee
	since $K_n$ tends to infinity this finishes the proof.
\end{proof}
\section{Upper and lower bound on weighted distances}\label{s:proof}
In this section we give the proofs of Theorems \ref{thm:mainthm} and \ref{thm:erasedthm}. We start with the main result as stated in Theorem \ref{thm:mainthm}, after that we give the proof of Theorem \ref{thm:erasedthm}. We start with the lower bound as stated in the following lemma:
\begin{lemma}[Lower bound on the weighted graph distance]
	\label{lem:lowerbound}
	Consider $\CMD$  satisfying Assumptions \ref{ass:degree-dist}-\ref{ass:tv} and let $u$ and $v$ be uniformly chosen from $[n]$. Suppose the edge lengths are i.i.d.\ with distribution function $\FL$ that satisfies \eqref{eq:weigthdist}. Then for all $\ve > 0$
	\be\label{eq:lower-prop} \lim_{n \to \infty} \bP\Bigg(\dl > (1-\ve)2\an\Bigg) = 1, \ee
	and for the hopcount \be\label{eq:lower-hop} \bP\left(\mathrm d_{H}(u,v) > (1-\ve)2 \log \log n / |\log (\tau-2)|\right) = 1.\ee
\end{lemma}

\begin{proof}
	We consider two uniformly chosen vertices $u$ and $v$. We do a BFS-exploration on both sides and by Lemma \ref{lem:lastgenexpl}, we can couple these explorations whp to two independent BPs until generation $\kappa_n(u), \kappa_n(v)$ respectively. We write $\Delta B_{\kappa_n(q)}(q)$ for the set of vertices distance  $\kappa_n(q)$ from vertex $q \in\{ u,v\}$, respectively. By the coupling, these explorations are disjoint whp. Since any path connecting $u,v$ must intersect $ \Delta B_{\kappa_n(u) }(u),  \Delta B_{\kappa_n(v) }(v)$, we have the following lower bounds on the weighted distance and the hopcount between $u,v$: 	\be\ba \label{eq:wd-lowerb} \dl &\geq d_L(u, \Delta B_{\kappa_n(u) }(u)) + d_L(v, \Delta B_{\kappa_n(v)}(v)),\\
	 \mathrm d_{H}(u,v) &\ge \kappa_n(u) + \kappa_n(v).\ea\ee
	Then, \eqref{eq:lower-prop} directly follows from the first inequality and  Lemma \ref{lem:lowerboundexpl}. 
	By \eqref{eq:wd-lowerb}, the result of the lemma follows by a union bound. For the hopcount, the second inequality combined with Lemma \ref{lem:lastgenexpl} yields \eqref{eq:lower-hop}, since $Y_{\kappa_n(q)}^{\sss{(q)}}$ converges in distribution.
\end{proof}
For the proof of the  upper bound  we use a proposition, similar to \cite[Proposition 2.1]{BarHofKom16}, which gives an upper bound on the path length between two vertices of a fixed degree of at least $K$. In our setting the vertices have a degree of at least $K_n$ with $K_n$ tending to infinity with $n$. We provide the adjusted proof since the adjustments are non-trivial.

\begin{proposition} \label{prop:pathlength}
	Consider $\CMD$ satisfying \eqref{eq:Fn} for all $x \in [\theta, n^{\al}]$ for some given $\theta \in \bR$ and some $\al>1/2$. Let $u_{K_n}$ be a vertex with degree at least $K_n$. Then, whp, there exists a path from $u_{K_n}$ to a vertex $u^\star$ with degree at least $n^{(\tau-2) \al}$ such that the degree $y_i(K_n)$ of the $i$th vertex on the path satisfies
	\be\label{eq:yikn} y_i(K_n) \geq \left(K_n^{1-\delta_n}\right)^{\left(\frac{1}{\tau - 2}\right)^i}, \ee
	with $\delta_n\to 0$ as $K_n\to \infty$. Whp, $i_{\max}$, the length of this path is at most
	\be i_{\max}\le \frac{\log \log n}{|\log(\tau-2)|} - \frac{\log \log K_n}{|\log(\tau-2)|}.\ee

\end{proposition}

\begin{proof}
	We shall denote the number of edges on the path from $u_{K_n}$ to $u^\star$ by $i_{\max}$ and we define the following sets of vertices 
	\be \label{eq:layer} \Gamma_{i} := \{v \in [n]: d_v \geq y_i(K_n)\} \ee 
	for some sequence $y_i(K_n)=:y_i$ to be determined shortly. $(\Gamma_i)_{i\le i_{\max}}$ can be seen as layers of the graph, where $i_{\max}$ is the maximal $i$ such that $\Gamma_i$ is non-empty.
	Our goal is to prove that there exists a sequence $y_i(K_n)$ such that the following holds: 
	\be \label{eq:totalprobnotconnect} \lim_{n \to \infty} \sum_{i = 0}^{i_{\max}} \bP \left( u_i \in \Gamma_i, u_i \nrightarrow \Gamma_{i+1} \mid d_{u_0} \geq K_n \right) = 0, \ee
	where $u_i$ is a vertex chosen from $\Gamma_{i}$ according to the size-biased distribution, equivalently,  vertex $u_i$ is the vertex that a uniformly chosen half-edge from $\Gamma_{i}$ is attached to. Conditioning on the total number of half edges  $\cH_n$ in $\CMD$, and $\cH_{y_i}$, the number of half-edges attached to vertices in the set $\Gamma_{i}$, by pairing the half-edges of a  vertex $w\in \Gamma_i$, we can pair at least $y_i/2$ half-edges before all the half-edges of $w$ are paired, and each of these half-edges is paired to a half-edge attached to a vertex in $\Gamma_{i+1}$ with probability at least $1-\cH_{y_{i+1}}/ \cH_n$. Thus,   	\be \label{eq:probnotconnect} \bP\left(w \in \Gamma_i, w \nrightarrow \Gamma_{i+1} \mid  \cH_{y_{i+1}}, \cH_n \right) \leq \left(1 - \frac{\cH_{y_{i+1}}}{\cH_n}\right)^{y_i/2}. \ee
	Note in particular that this bounds holds when the vertex is chosen randomly from $\Gamma_i$ in a way that does not take into account its connections, in particular it holds when $w$ is chosen size-biasedly from $\Gamma_i$.
	Since any vertex in $\Gamma_{i}$ has degree larger than $y_i$ and $|\Gamma_{i}| = n(1-F_n(y_i))$, $\cH_{y_i} \geq y_in(1 - F_n(y_i))$.
	Under Assumptions \ref{ass:degree-dist}, \ref{ass:tv}, $\cH_n \leq \varphi n $ for some $\varphi \in \bR$, thus, we have by \eqref{eq:probnotconnect}
	\be \label{eq:probnotconnect2} \bP\left(w \in \Gamma_i, w \nrightarrow \Gamma_{i+1} \right) \leq \exp \left\{-\frac{y_iy_{i+1}(1 - F_n(y_{i+1})}{2 \varphi}\right\} . \ee
	For now we focus on the term in the exponent. Using \eqref{eq:Fn}, we lower bound 
	\be \label{eq:probnotconnectlower} \frac{y_iy_{i+1}(1 - F_n(y_{i+1}))}{2 \varphi} \geq \tilde{c} y_i y_{i+1}^{2-\tau} \mathrm{e}^{-C(\log y_{i+1})^{\gamma}} = \tilde{c}y_i y_{i+1}^{2-\tau-C(\log y_{i+1})^{\gamma-1}}, \ee
	with C defined in \eqref{eq:Fn} and $\tilde{c}$ some positive constant. Now we would like to choose the sequence $y_i=y_i(K_n)$ such that \eqref{eq:probnotconnect2} converges to zero in particular that \eqref{eq:totalprobnotconnect} holds. We claim that this holds when $y_i$ is given by the following recursion
	\be \label{eq:yrecursion} y_0 = K_n, \;\;\;\;\;\;\;\;\; y_{i+1} = y_i^{(\tau-2+D(\log y_i)^{\gamma-1})^{-1}} \ee
	with $D>0$ defined later. Note that  for sufficiently large $K_n$, since $\gamma < 1$, 
	\be \tau -2 +D(\log y_0)^{\gamma - 1} < 1 \ee
	Now let $\varkappa_n = D(\log K_n)^{\gamma - 1}$, then
	\be \label{eq:yinequality} y_{i+1} \geq y_i^{(2-\tau+\varkappa_n)^{-1}} \geq \ldots \geq K_n^{(\tau - 2+\varkappa_n)^{-i}}. \ee
	We use the recursion relation of \eqref{eq:yrecursion} in \eqref{eq:probnotconnectlower} 
	\be \tilde{c}y_i y_{i+1}^{2-\tau-C(\log y_{i+1})^{\gamma-1}} = \tilde{c} y_i^{\frac{2-\tau-C(\log y_{i+1})^{\gamma-1}}{\tau - 2 + D(\log y_i)^{\gamma-1}} + 1} \geq \tilde{c} y_i^{\frac{D(\log y_i)^{\gamma-1}-C(\log y_{i+1})^{\gamma-1}}{\tau - 2 + \varkappa_n}}\ee
Choose $D \geq 2C$ and use that the sequence $y_i$ is increasing and the lower bound in \eqref{eq:yinequality}, then
	\be \tilde{c} y_i^{\frac{D(\log y_i)^{\gamma-1}-C(\log y_{i+1})^{\gamma-1}}{\tau - 2 + \varkappa_n}} \geq \tilde{c} \exp\left\{\frac{C(\log y_i)^{\gamma}}{\tau-2-\varkappa_n}\right\} \geq \tilde{c}\exp\Big\{\frac{\wit C(\log K_n)^\gamma}{(\tau - 2 + \varkappa_n)^{i\gamma}}\Big\}, \ee
	with $\wit C=C/(\tau-2-\varkappa_n)$.
	Combining everything from \eqref{eq:probnotconnectlower}, we can us this lower bound in the exponent on the rhs of \eqref{eq:probnotconnect2}, and, since $\tau-2+\varkappa_n < 1$, the rhs of \eqref{eq:probnotconnect2} is summable in $i$. Summing the lhs of \eqref{eq:probnotconnect2} over $i$ and then use the above bound we obtain
	\be \label{eq:probconnectresult} \sum_{i = 0}^{\infty} \bP \left( u_i \in \Gamma_i, u_i \nrightarrow \Gamma_{i+1} \mid d_{u_0} \geq K_n \right) \leq \hat{C} \exp\Big\{-\tilde{c}\exp\big\{\frac{C(\log K_n)^\gamma}{(\tau - 2 + \varkappa_n)^{\gamma}}\big\}\Big\}, \ee
	which tends to zero with n as $K_n$ tends to infinity with n. This result yields the statement of $\eqref{eq:totalprobnotconnect}$. Using the result of \cite[Lemma 2.6]{BarHofKom16}, the lower bound in \eqref{eq:yinequality} can be improved to
	\be \label{y_ilowerbound} y_i \geq \left(y_0^{1-\delta_n}\right)^{(\tau-2)^{-i}},\ee
	with $\delta_n\to 0$ as $K_n\to \infty$.\footnote{From the proof of \cite[Lemma 2.4]{BarHofKom16}, it is immediate that $\delta_n\le \wit D(\log K_n)^{\gamma-1}$ for some constant $\wit D>0$.}
	
	The path in the statement of the lemma is then constructed as follows, starting from the first vertex $u_0 = u_{K_n}$. By the first term in the sum in \eqref{eq:totalprobnotconnect}, $u_0$ is whp connected to at least one vertex in $\Gamma_1$. By the fact that the pairs of the half edges of $u_0$ are chosen uniformly, $u_1$ is a vertex attached to a uniformly chosen half-edge in $\Gamma_1$. As a result, $u_1$ is chosen according to the size-biased distribution within $\Gamma_1$. Then we iterate this procedure to obtain $u_2, u_3, \ldots$ in $\Gamma_2, \Gamma_3, \ldots$ until we reach a vertex of degree at least $n^{\alpha}$ for $\alpha = (\tau-2)(1+\zeta)/(\tau-1)$. The constructed path uses at most all layers so the number of layers is an upper bound on the length of the path from $u_{K_n}$ to $u^\star$. The last layer that is nonempty is then $\Gamma_{i_{\max}}$ with $i_{\max}$ is the largest integer with
		\be \label{eq:max-i-deg}K_n^{(1-\delta_n)\left(\tau - 2\right)^{-i_{\max}}} \le  n^{\alpha}. \ee 
	Since $\delta_n\to 0, \al < 1, $ the following upper bound then holds
		\be\label{eq:imax-2} i_{\max} \leq \frac{\log \log n - \log \log K_n + \log \alpha-\log (1-\delta_n)}{|\log(\tau-2)|} \leq \frac{\log \log n - \log \log K_n}{|\log(\tau-2)|}. \ee
		We yet have to show that $y_{i_{\max}}\ge n^{\al(\tau-2)}$.
For this, elementary rearrangement yields that the lhs of \eqref{eq:max-i-deg} equals $n^{\al(\tau-2)^{\beta}}$, with $\beta\in[0,1)$ the fractional part of the middle term in \eqref{eq:imax-2}. This finishes the proof.
\end{proof}
	
\begin{lemma}[Upper bound on the weighted graph distance]
	\label{lem:upperbound}
	Consider $\CMDL$ satisfying Assumptions \ref{ass:degree-dist}-\ref{ass:tv} and $u,v$ two uniformly chosen vertices. Suppose the edge weights are i.i.d.\ from $\FL$  that satisfies \eqref{eq:weigthdist}. Then for all $\ve > 0$
	\be \label{eq:lem5.3} \lim_{n \to \infty} \bP\Bigg( d_L(u, v) < (1 + \ve) 2\an\Bigg) = 1. \ee
Further, there exists a path between $u,v$ with at most $(1+\ve)2\log \log n/(\tau-2)$ edges and having total length at most $(1 + \ve) 2\an$.
\end{lemma}

\begin{proof} For brevity let $a_n:=\an$.	
First we construct the initial segments of the connecting path from both ends from $u,v$, as described heuristically in Section \ref{s:proof-overview-upper}. Let $M_n$ be as in Lemma \ref{lem:upperboundexpl}, with $\ve_1:=\ve/3$. Then, consider any vertex $w$ of graph distance $M_n$ away in $\CMD$ from $q\in\{u,v\}$, chosen \emph{independently} of $(L_e)_e$. Then, since the edge-lengths in $\CMDL$ are i.i.d. on the edges of the path from $q$ to $w$, by Lemma \ref{lem:upperboundexpl}, $\mathrm d_L(q, w) \le a_n\ve/3$ in $\CMDL$ whp. As a result of Lemma \ref{lem:generation-Kn}, for any $M$ with $M|\log (\tau-2)|>1$, at graph distance $M \log \log \wit K_n$ away from $q\in\{u,v\}$, there is at least one vertex with degree $\wit K_n$ in $\CMD$ whp. Thus, by defining  $ \wit K_n$ via  $M_n=M \log \log \wit K_n$, (equivalently, $\wit K_n:= \exp\big\{\exp\{M_n/M\}\big\}$), we find vertices with degree at least $\wit K_n$ at graph distance $M_n$ away from $q\in\{u,v\}$, whp.  Then, pick $q_{K_n}$ for  $q\in \{u,v\}$ in an arbitrary way that is independent of $(L_e)_e$. Then, the previous argument applies and whp,
\be\label{eq:qqn}\mathrm d_L(q, q_{K_n}) \le a_n\ve/3 \ee
 in $\CMDL$ for $q\in\{u,v\}$.

Next we connect $u_{K_n}, v_{K_n}$ using degree-dependent percolation. When applying edge-dependent percolation (as in Def.~\ref{def:perc2}) on $\CMDL$, we can use the edge-lengths $(L_e)_e$ as auxiliary variables to decide which edge to keep. Namely, we keep edge $e$ iff $L_e\le \xi(d, d')$, with $\xi(d,d')$ satisfying $\Pv(L\le \xi(d,d'))=p(d)p(d')$. By Corollary \ref{cor:equalitydist}, we can consider the percolated (sub)graph as an instance of a configuration model where the new degree sequence is $\bf{d^r}$. We yet have to specify the percolation function that we use. For some $c>0,\eta\in(0,1)$ to be determined later, let 
\be \label{eq:percfunction} p(d) = \exp\{ -c(\log d)^\eta\}. \ee
The conditions of Lemma \ref{lem:degreeafterperc} apply, thus, with $K_n$ as in \eqref{eq:Kn}, $d_{q_{K_n}}^r\ge  K_n$ whp for $q\in\{u,v\}$. Further, the conditions of Lemma \ref{lem:percdegrees} are also satisfied, thus the $\bf{d}^r$ sequence obtained after percolation still satisfies Assumption \ref{ass:degree-dist} (except the condition on the minimal degree being at least $2$). 
Hence, following Proposition \ref{prop:pathlength}, we construct a path connecting $u_{K_n}, v_{K_n}$ in the \emph{percolated graph}, with good control on the (percolated) degrees along the path.

For $q\in \{u,v\}$, we use the constructed path as described in Proposition \ref{prop:pathlength} starting from  $q_{K_n}$ to reach a vertex $q^\star$ with  $d^r_{q^\star}\ge n^{\al(\tau-2)}$.  A lower bound on the degree of the $i$th vertex on this path is $y_i=y_i(K_n)$ given in \eqref{eq:yikn}. Since $p(d)$ is monotone decreasing, $\xi(d, d')$ is non-increasing in both variables. Thus, the edge-lengths on the constructed path are at most $\xi(y_{i},y_{i+1})$ for $i = 0, 1, \ldots, i_{\max}-1$. Hence, for $q\in \{u,v\}$
	\be \label{eq:upperhalfpath} d_L(q_{K_n}, q^\star) \leq \sum_{i = 0}^{i_{\max}-1}\xi(y_{i},y_{i+1}). \ee 
Next we connect the two high-degree vertices $u^\star$ and $v^\star$. Let us denote choose two vertices $w_1^\star, w_2^\star$  with degrees  is at least $n^{1/2+\delta}$ for $\delta\in (0, \al-1/2)$ arbitrary but fixed. Recall that  $\CH_n$ stands for the total number of half-edges, and is at least some constant $\phi n$ under Assumption \ref{ass:degree-dist} even without the minimal degree assumption.  
Fix $\delta\in (0, \al-1/2)$ and write $\Lambda_{1/2+\delta}:=\{w: d_w^r\ge  n^{1/2+\delta}\}$, as well as $\CH_{1/2+\delta}:=\sum_{w\in \Lambda_{1/2+\delta}} d_w^r$. 
Then, following \eqref{eq:probnotconnect}-\eqref{eq:probnotconnectlower},
	\be \label{eq:conn-bound} \bP\left( q^\star \not \leftrightarrow \Lambda_{1/2+\delta} \mid \CH_n, \CH_{1/2+\delta}\right) \leq \left(1 - \frac{\CH_{1/2+\delta}}{\cH_n}\right)^{n^{\al(\tau-2)}/2} \leq \exp\left\{-c n^{\al(\tau-2) + (2-\tau)(1/2+\delta) -o(1)}\right\}, \ee
	which tends to zero as $n\to \infty$ since $1/2+\delta<\al$. Thus, we can find vertices  $u^{\star\star}, v^{\star\star}\in \Lambda_{1/2+\delta}$ 
 such that $(q^{\star}, q^{\star\star})$ are kept edges the percolated graph whp.
Finally, we show that the edge $(u^{\star\star}, v^{\star\star})$ is also present whp in the percolated graph. 
\be \label{eq:conn-bound} \bP\left( u^{\star\star} \not \leftrightarrow v^{\star\star} \mid \CH_n\right) \leq \left(1 - \frac{n^{1/2+\delta}}{\cH_n}\right)^{n^{1/2+\delta}/2} \leq \exp\left\{-c n^{(1+\delta) -1}\right\}, \ee
which tends to zero as $n\to \infty$.
	By the monotonicity of $\xi$, whp, the vertices $u^\star$ and $v^\star$ are connected via at most 3 edges with length at most
	\be \label{eq:connectingtwopaths} d_L(u^\star, v^\star) \leq 3 \xi(n^\alpha(\tau-2), n^{1/2}). \ee Combining \eqref{eq:qqn},\eqref{eq:upperhalfpath} and \eqref{eq:connectingtwopaths}, we arrive at \eqref{eq:upper-heuristic}.
	In what follows we show that the rhs of \eqref{eq:upperhalfpath} is at most $(1+\ve/3)a_n$.  By the definition of edge-percolation in Def.~\ref{def:perc2}, we keep an edge connecting vertices with degrees $y_i, y_{i+1}$ with probability at most $p(y_i)p(y_{i+1})$. Using the form $p(d)$ from \eqref{eq:percfunction} and its monotonicity, and the lower bound on $y_i$ from Prop.~\ref{prop:pathlength},
		\be\label{eq:pyi} p(y_i)p(y_{i+1}) \le  \exp\{ -c(\log y_i)^\eta - (\log y_{i+1})^\eta\} \leq \exp\left\{-\wit c (\tau-2)^{-\eta (i+1)}(\log K_n^{1-\delta_n})^\eta\right\}, \ee
		with $\wit c=c (1+(\tau-2)^{\eta}).$ 
Recall that we keep an edge $e$ between vertices with degrees $d, d'$ iff its edge-length is at most $\xi(d, d')$ in the edge-percolation. This gives us  $\xi(d, d')=F_L^{(-1)}(p(d)p(d'))$, and, by monotonicity again, from \eqref{eq:pyi} it follows that
	\be \xi(y_{i},y_{i+1})\leq F_L^{\sss{(-1)}}\Big(\exp\big\{-c(\log K_n^{1-\delta_n})^\eta(\tau-2)^{-\eta (i+1)}\big\}\Big). \ee
	Combining \eqref{eq:pyi} with the bound on $i_{\max}$ from  Proposition \ref{prop:pathlength}, \eqref{eq:upperhalfpath} can be bounded above as
	\be d_L(q_{K_n}, q^\star) \leq \sum_{i = 1}^{\left\lfloor\frac{\log (\log n/\log K_n)}{|\log(\tau-2)|}\right\rfloor}F_L^{\sss{(-1)}}\Big(\exp\big\{-c(\log K_n^{1-\delta_n})^\eta(\tau-2)^{-\eta i}\big\}\Big). \ee
		Similar to the proof of Lemma \ref{lem:lowerboundexpl}, we need to transform the rhs to the desired form in \eqref{eq:lem5.3}. Using similar bounds as in \eqref{eq:tranformationineq}, we rewrite the sum to an integral, change variables as $(\tau-2)^{-\eta x}(\log K_n^{1-\delta_n})^\eta =: (\tau-2)^{-y}$, and change the integral back to a sum. This operation shifts the summation boundaries by $\eta \log\log K_n / |\log(\tau-2)| $ and multiplies the whole sum by $\eta$. We obtain
	\be
	 d_L(q_{K_n}, q^\star)	\leq \frac{1}{\eta} \sum_{i=\lfloor\eta \frac{\log \log K_n}{|\log(\tau-2)|}\rfloor}^{\lceil\eta \frac{\log \log n}{|\log(\tau-2)|}\rceil}F_L^{\sss{(-1)}}(\e^{-1/(\tau-2)^i}). \ee
	By choosing $\eta\in(0,1)$ in \eqref{eq:percfunction} such that $1/\eta < 1+\ve/3$ so we obtain that
	\be \label{eq:upperpathbound} d_L(q_{K_n}, q^\star) \leq (1+\ve/3) \an. \ee
	Finally, it is not hard to see that $3 \xi(n^\alpha(\tau-2), n^{1/2})\le  a_n\ve/3$ holds as well for all large enough $n$. Combining everything, we arrive at 
	\be\ba \mathrm{d}_L(u,v) &\leq  \sum_{q\in \{u,v\}} \big(\mathrm{d}_L(q, q_{K_n})+ 
\mathrm{d}_L(q_{K_n}, q^\star)\big) + \mathrm{d}_L(u^\star, v^\star)\\
	&\le 2 a_n (\ve/3 + (1+\ve/3)) + a_n\ve/3 \le 2a_n (1+\ve).\ea\ee
This finishes the proof of \eqref{eq:lem5.3}. For the second statement, recall that for some $M\ge 1/|\log (\tau-2)|$, $M_n=M\log \log(\wit K_n)$ and  note that the number of edges on the constructed path is at most 
\be 2M\log \log \wit K_n+2\frac{\log \log n -\log \log K_n}{|\log (\tau-2)|}+3,\ee
where the relation between $\wit K_n$ and $K_n$ is described in Lemma \ref{lem:degreeafterperc} in \eqref{eq:Kn}. From \eqref{eq:Kn} it is elementary to check that for all $n$ large enough
\be \log \log K_n= \log \log \wit K_n +\log (1-c(\log \wit K_n)^{\eta-1})=\log \log \wit K_n +o(1),\ee
thus, writing $M:=(1+z)/|\log (\tau-2)|$ for some $z>0$, the number of edges in the constructed path is at most
\be \frac{2\log \log n + 2z \log \log \wit K_n + o(1)}{|\log (\tau-2)|}+3\le (1+\ve) \frac{2\log \log n}{|\log (\tau-2)|}, \ee
as desired.
\end{proof}
\begin{proof}[Proof of Theorem \ref{thm:mainthm}]
Lemma \ref{lem:lowerbound} states the proof of the lower bound and Lemma \ref{lem:upperbound} the proof of the upper bound. These combined prove the statement of the theorem.
\end{proof}

\subsection{Erased configuration model}
In this section we prove Theorem \ref{thm:erasedthm}. 

\begin{proof}[Proof of Theorem \ref{thm:erasedthm}, lower bound]
The strategy of the proof is the following: first we show that the lower bound is also valid in the erased model. Then, we show that the constructed paths in the proof of the upper bound between vertices $q, q_{K_n}$ and $q_{K_n}, q^\star$ are whp simple for $q\in \{u,v\}$, and as a result they survive the erasing procedure whp. Finally, we connect $u^\star, v^\star$ in the erased model in some other way than that in the original model.

First we start with the lower bound. The proof of Lemma \ref{lem:lowerbound} consists of a BFS exploration around the two vertices $u$ and $v$. These explorations can whp be coupled to two BP tress and therefore all edges within these trees are whp simple. So this lemma remains valid after erasure and thus the lower bound follows both for the weighted distance as well as for the hopcount.
\end{proof}
In the proof of the upper bound we again use a coupling to BP trees  to find $u_{K_n}, v_{K_n}$. Thus, the path
between $q, q_{K_n}$ is again whp simple and thus it survives erasure. Next we investigate the constructed path between $q_{K_n}, q^\star$. This path is constructed in the percolated graph. The erasure happens \emph{before} the degree-dependent percolation, so edges of the path constructed in Proposition \ref{prop:pathlength} could in principle be deleted earlier in the erasure procedure.  We show that the edges on the constructed path were not part of a multiple edge whp, meaning that they were whp not erased before. 
For this, we state a lemma that gives a bound on the original degree of a vertex, given its percolated degree $d^r$. This lemma is the `reverse' of Lemma \ref{lem:degreeafterperc}.
\begin{claim}[Degree after percolation vs original degree]
	\label{cl:orig-deg}
 Apply  half-edge percolation as described in Definition \ref{def:perc1} with percolation function $p(d)$  satisfying \eqref{eq:perccondition} on $\CMD$. Let  $\omega(n)$ be an arbitrary sequence that tends to infinity with $n$.  Let $s(x)$ be defined as in \eqref{eq:y(x)}.
	 Then, for a vertex $w\in \CMD$, 
	 \be\label{eq:bin-conc-2} \Pv( d_w \ge s(x) \mid d_w^r\le x) \le c\exp\{ - x/4\}\ee
	 for some $c>0$.
\end{claim}
\begin{proof}
The proof directly follows from Bayes' theorem applied to the lhs of \eqref{eq:bin-conc-2}, and following the the calculations between \eqref{eq:y(x)} and \eqref{eq:bin-conc}. 
\end{proof}
\begin{lemma}[No multiple edges on the path $q_{K_n}, q^\star$] \label{lem:uniquepath}
	Let $u_i$ and $u_{i+1}$ be two consecutive vertices on the constructed path between $u_{K_n}, u^\star$  in Proposition \ref{prop:pathlength} and $i = 0, \ldots i^{\max}-1$, then
	\be \label{eq:uniqueedges} \lim_{n \to \infty} \bP( \geq 2 \text{ edges connecting } u_i \leftrightarrow u_{i+1} | \geq 1 \text{ edge connecting } u_i \leftrightarrow u_{i+1}) = 0. \ee
\end{lemma}
\begin{proof}
Note that the path in  Prop.~\ref{prop:pathlength} is later, in the proof of Theorem \ref{thm:mainthm} is  constructed in the \emph{percolated graph}. Thus $u_i$, the $i$th vertex on this path has percolated degree at least as in \eqref{eq:yikn}. Without loss of generality we can assume that 
\be\label{eq:perc-upper} d_{u_i}^r \le \left(K_n^{1-\delta_n}\right)^{\left(\frac{1}{\tau - 2}\right)^{i+1}}=:y_{i+1}, \ee
since otherwise the path has `jumped' a layer and one can consider the path to be shorter by an edge. 
	Recall that $d_{u_{i_{\max}}}\ge n^{\al(\tau-2)}$ holds as well.  Applying Claim \ref{cl:orig-deg} on $(u_i)_{i\le i_{\max}-1}$, using the upper bound in \eqref{eq:perc-upper},
\be \label{eq:orig-deg-upper}\Pv\left( \exists i\le i_{\max}\!-\!1, d_{u_i} \ge s(y_{i+1}) \mid  d_{u_i}^r \le y_{i+1} \forall i \le i_{\max}\!-\!1  \right) \le \sum_{i=0}^{i_{\max}-1} c \exp\{ -y_{i+1}/4 \}\ee
which tends to zero with $n$ since it is a constant times the first term.

We can rewrite the probability in \eqref{eq:uniqueedges} as 
	\be \label{eq:uniquefrac} \frac{1 - \bP(1 \text{ edge } u_i \leftrightarrow u_{i+1}) - \bP( u_i \not\leftrightarrow u_{i+1})}{1 - \bP( u_i \not\leftrightarrow u_{i+1})}. \ee
	We investigate the probabilities in \eqref{eq:uniquefrac} separately starting with the probability that there is exactly one edge between those two vertices. We lower bound the probability that precisely the $j$th half-edge of $u_i$ connects to $u_{i+1}$, and the others do not. Note that for the $k$th  half-edge the probability of not connecting to $u_{i+1}$ is  at least $(\he - d_{u_{i+1}} -2(k-1))/(\he - 2(k-1) - 1)\ge 1- d_{v_{i+1}}/\he$. Thus
\be \label{eq:prob1}
	\bP(1 \text{ edge } u_i \leftrightarrow u_{i+1}) \geq \sum_{j=1}^{d_{u_i}} \frac{d_{u_{i+1}}}{\cH_n-2d_{u_{i}}} \prod_{k=1}^{d_{u_{i}}-1} \left(1 - \frac{d_{u_{i+1}}}{\he}\right) 
	\geq \frac{d_{u_{i+1}}d_{u_{i}}}{\cH_n-2d_{u_{i}}}\left(1 - \frac{d_{u_{i+1}}}{\he}\right)^{d_{u_{i}}}.
	\ee
	Next we bound the probability that there is no edge between two consecutive vertices, both from above and below. 
	\begin{align} \label{eq:prob0up}
	\bP( u_i \not\leftrightarrow u_{i+1}) & \leq \prod_{k=1}^{\fl{d_{u_{i}}/2}} \left(1 - \frac{d_{u_{i+1}}}{\he-2d_{u_{i}}}\right) = \left(1 - \frac{d_{u_{i+1}}}{\he-2d_{u_{i}}}\right)^{d_{u_i}/2},\\ \label{eq:prob0low}
	\bP(u_i \not\leftrightarrow u_{i+1}) & \geq \prod_{k=1}^{d_{u_{i}}} \left(1 - \frac{d_{u_{i+1}}}{\he}\right) = \left(1 - \frac{d_{u_{i+1}}}{\he}\right)^{d_{u_i}}.
	\end{align}
	Using series expansion for equations \eqref{eq:prob1}-- \eqref{eq:prob0low}, we obtain an upper bound on \eqref{eq:uniquefrac}:
	\be \label{eq:uniineq} \frac{1 - \bP(1 \text{ edge } u_i \leftrightarrow u_{i+1}) - \bP( u_i \not\leftrightarrow u_{i+1})}{1 - \bP( u_i \not\leftrightarrow u_{i+1})} \leq \frac{2d_{u_{i}}d_{u_{i+1}}}{\he}\Big/\left(1-\frac{d_{u_{i}}d_{u_{i+1}}}{4(\he-d_{u_i})}\right) .\ee
	By \eqref{eq:orig-deg-upper}, whp,  $d_{u_{i+1}} \leq s(y_{i+1})$, and further, by the definition of $i_{\max}$ in \eqref{eq:max-i-deg}, and $s(\cdot)$ in \eqref{eq:y(x)},    $s(y_{i_{\max}-k}) \le n^{\al(\tau-2)^k(1+o(1))}$ for $k\in \{1,2\}$.  Thus, for all $i\le i_{\max}\!-\!2$, whp
	\be \nonumber \frac{d_{u_{i}}d_{u_{i+1}}}{\he} \leq c \frac{s(y_{i_{\max}-2}) s(y_{i_{\max}-1})   }{n} \leq n^{\al(\tau-2)(\tau-1)(1+o(1))-1} .\ee
	The rhs converges to zero as $n$ tends to infinity as long as $\al<((\tau-2)(\tau-1))^{-1}$, which we have assumed in Assumption \ref{ass:degree-dist}.
\end{proof}
\begin{proof}[Proof of Theorem \ref{thm:erasedthm}, upper bound]
As mentioned before, we construct a path in $\ECMDL$ to connect $u,v$. For this it is enough to construct a path with all its edges begin simple edges in $\CMDL$.  This path has a huge overlap with the path in the upper bound of Theorem \ref{thm:mainthm}. Namely, the segments  between $u, u_{K_n}$ and $v, v_{K_n}$ are whp using simple edges by the coupling to BP trees.
The segments between $u_{K_n}, u_{i_{\max}-1}$ and $v_{K_n}, v_{i_{\max}-1}$ are whp using simple  edges again so they survives erasure. 
Next we connect $u_{i_{\max}-1}$ to $v_{i_{\max}-1}$. Note that the constructed path in $\CMDL$ might use multiple edges so we need a different connecting path. 
However, $q_{i_{\max}-1}$ for $q\in\{u,v\}$ are vertices with degree at least $n^{\al(\tau-2)(1+o(1))}$. In the proof of Theorem \ref{thm:mainthm}, we created a 3-hop connection between $u_{i_{\max}}=u^\star$ and $v_{i_{\max}}=v^\star$ in the \emph{percolated} graph, see \eqref{eq:conn-bound}--\eqref{eq:connectingtwopaths}.
When we erase a multiple edge, we keep one edge independently of its edge-length. Thus, from every multiple edge at least one edge remains. Hence, an analogous construction as in \eqref{eq:conn-bound}--\eqref{eq:connectingtwopaths} can be repeated, not for the percolated graph but for the original graph, developing a 5-hop connection between $u_{i_{\max}-1}, v_{i_{\max}-1}$. The edge-lengths on this path are simply i.i.d. copies of $L$.
Thus,
\be \label{eq:erasedineq} d_L^e(u, v) \leq d(u, u_{K_n}) + d(v, v_{K_n}) +  2\sum_{i = 0}^{i_{\max}-2} \xi(y_i,y_{i+1}) + \sum_{i = 1}^{5} L_i .\ee
For all $\ve > 0$ 
\be \nonumber \lim_{n \to \infty} \bP\Bigg( \sum_{i = 1}^{5} L_i \leq \ve/3 \; \an \Bigg) = 0 .\ee
Then we treat the terms in \eqref{eq:erasedineq} similarly as we did in the proof of Theorem \ref{thm:mainthm} (see \eqref{eq:upperhalfpath} and \eqref{eq:pyi}--\eqref{eq:upperpathbound}) finishes the proof.
\end{proof}

\bibliographystyle{abbrv}
\bibliography{refsthesis}

\begin{thebibliography}{10}

\bibitem{Acha2006}
S.~Achard, R.~Salvador, B.~Whitcher, J.~Suckling, and E.~Bullmore.
\newblock A resilient, low-frequency, small-world human brain functional
  network with highly connected association cortical hubs.
\newblock {\em The Journal of Neuroscience}, 26(1):63--72, 2006.

\bibitem{AieBon08}
W.~Aiello, A.~Bonato, C.~Cooper, J.~Janssen, and P.~Pra{\l}at.
\newblock A spatial web graph model with local influence regions.
\newblock {\em Internet Mathematics}, 5(1-2):175--196, 2008.

\bibitem{AlbBar02}
R.~Albert and A.-L. Barab{\'a}si.
\newblock Statistical mechanics of complex networks.
\newblock {\em Reviews of modern physics}, 74(1):47, 2002.

\bibitem{AmiDevGriOlv13}
O.~Amini, L.~Devroye, S.~Griffiths, N.~Olver, et~al.
\newblock On explosions in heavy-tailed branching random walks.
\newblock {\em The Annals of Probability}, 41(3B):1864--1899, 2013.

\bibitem{Bara1999}
A.-L. Barab{\'a}si and R.~Albert.
\newblock Emergence of scaling in random networks.
\newblock {\em science}, 286(5439):509--512, 1999.

\bibitem{BarAlbJeo00}
A.-L. Barab{\'a}si, R.~Albert, and H.~Jeong.
\newblock Scale-free characteristics of random networks: the topology of the
  world-wide web.
\newblock {\em Physica A: statistical mechanics and its applications},
  281(1):69--77, 2000.

\bibitem{BarHofKom16}
E.~Baroni, R.~v.~d. Hofstad, and J.~Komj{\'a}thy.
\newblock Tight fluctuations of weight-distances in random graphs with
  infinite-variance degrees.
\newblock {\em arXiv preprint arXiv:1609.07269}, 2016.

\bibitem{BarHofKomdist}
E.~Baroni, R.~v.~d. Hofstad, and J.~Komj{\'a}thy.
\newblock Nonuniversality of weighted random graphs with infinite variance
  degree.
\newblock {\em Journal of Applied Probability}, 54(1):146--164, 2017.

\bibitem{BenCan78}
E.~A. Bender and E.~R. Canfield.
\newblock The asymptotic number of labeled graphs with given degree sequences.
\newblock {\em Journal of Combinatorial Theory, Series A}, 24(3):296--307,
  1978.

\bibitem{BhaHofHoo10}
S.~Bhamidi, R.~v.~d. Hofstad, and G.~Hooghiemstra.
\newblock Extreme value theory, poisson-dirichlet distributions, and first
  passage percolation on random networks.
\newblock {\em Advances in applied probability}, 42(3):706--738, 2010.

\bibitem{BHH10}
S.~Bhamidi, R.~v.~d. Hofstad, and G.~Hooghiemstra.
\newblock First passage percolation on random graphs with finite mean degrees.
\newblock {\em Ann. Appl. Probab.}, 20(5):1907--1965, 2010.

\bibitem{BHH11}
S.~Bhamidi, R.~v.~d. Hofstad, and G.~Hooghiemstra.
\newblock First passage percolation on the {E}rd{\H o}s-{R}{\'e}nyi random
  graph.
\newblock {\em Combinatorics, Probability and Computing}, 20:683--707, 2011.

\bibitem{bhamidi2017}
S.~Bhamidi, R.~v.~d. Hofstad, and G.~Hooghiemstra.
\newblock Universality for first passage percolation on sparse random graphs.
\newblock {\em Ann. Probab.}, 45(4):2568--2630, 07 2017.

\bibitem{BogPapKri10}
M.~Bogun{\'a}, F.~Papadopoulos, and D.~Krioukov.
\newblock Sustaining the {I}nternet with hyperbolic mapping.
\newblock {\em Nature Communications}, 1(62), 2010.

\bibitem{Boll80}
B.~Bollob{\'a}s.
\newblock A probabilistic proof of an asymptotic formula for the number of
  labelled regular graphs.
\newblock {\em European Journal of Combinatorics}, 1(4):311 -- 316, 1980.

\bibitem{BriKeuLen15}
K.~Bringmann, R.~Keusch, and J.~Lengler.
\newblock Geometric inhomogeneous random graphs.
\newblock arXiv:1511.00576, 2015.

\bibitem{BulSpo09}
E.~Bullmore and O.~Sporns.
\newblock Complex brain networks: graph theoretical analysis of structural and
  functional systems.
\newblock {\em Nature reviews. Neuroscience}, 10(3):186, 2009.

\bibitem{ChuLu01}
F.~Chung and L.~Lu.
\newblock The diameter of sparse random graphs.
\newblock {\em Adv. in Appl. Math.}, {\bf 26}(4):257--279, 2001.

\bibitem{D78}
P.~L. Davies.
\newblock The simple branching process: a note on convergence when the mean is
  infinite.
\newblock {\em J. Appl. Probab.}, 15(3):466--480, 1978.

\bibitem{DeiHof13}
M.~Deijfen, R.~v.~d. Hofstad, and G.~Hooghiemstra.
\newblock Scale-free percolation.
\newblock {\em Annales de l'Institut Henri Poincar{\'e}, Probabilit{\'e}s et
  Statistiques}, 49(3):817--838, 2013.

\bibitem{EskHofHoo08}
H.~v.~d. Esker, R.~v.~d. Hofstad, and G.~Hooghiemstra.
\newblock Universality for the distance in finite variance random graphs.
\newblock {\em Journal of Statistical Physics}, 133(1):169--202, 2008.

\bibitem{EskHofHooZna05}
H.~v.~d. Esker, R.~v.~d. Hofstad, G.~Hooghiemstra, and D.~Znamenski.
\newblock Distances in random graphs with infinite mean degrees.
\newblock {\em Extremes}, 8(3):111--141, 2005.

\bibitem{Falo1999}
M.~Faloutsos, P.~Faloutsos, and C.~Faloutsos.
\newblock On power-law relationships of the internet topology.
\newblock {\em ACM SIGCOMM computer communication review}, 29(4):251--262,
  1999.

\bibitem{HamWel65}
J.~M. Hammersley and D.~Welsh.
\newblock First-passage percolation, subadditive processes, stochastic
  networks, and generalized renewal theory.
\newblock In {\em Bernoulli 1713 Bayes 1763 Laplace 1813}, pages 61--110.
  Springer, 1965.

\bibitem{H10}
R.~v.~d. Hofstad.
\newblock {\em Random Graphs and Complex Networks, Vol. I}.
\newblock Cambridge University Press, 2016.

\bibitem{H102}
R.~v.~d. Hofstad.
\newblock {\em Random Graphs and Complex Networks, Vol. II}.
\newblock Cambridge University Press, 2016.
\newblock to appear.

\bibitem{HofHooVan05a}
R.~v.~d. Hofstad, G.~Hooghiemstra, and P.~Van~Mieghem.
\newblock Distances in random graphs with finite variance degrees.
\newblock {\em Random Structures Algorithms}, {\bf 27}(1):76--123, (2005).

\bibitem{HHZ07}
R.~v.~d. Hofstad, G.~Hooghiemstra, and D.~Znamenski.
\newblock Distances in random graphs with finite mean and infinite variance
  degrees.
\newblock {\em Electron. J. Probab.}, 12:no. 25, 703--766, 2007.

\bibitem{HofKom16}
R.~v.~d. Hofstad and J.~Komj{\'a}thy.
\newblock When is a scale-free graph ultra-small?
\newblock {\em Journal of Statistical Physics}, 169(2):223--264, Oct 2017.

\bibitem{JacMor13}
E.~Jacob and P.~M{\"o}rters.
\newblock A spatial preferential attachment model with local clustering.
\newblock In {\em International Workshop on Algorithms and Models for the
  Web-Graph}, pages 14--25. Springer, 2013.

\bibitem{Jan09}
S.~Janson.
\newblock On percolation in random graphs with given vertex degrees.
\newblock {\em Electronic Journal of Probability}, 14:86--118, 2009.

\bibitem{JanLuc09}
S.~Janson and M.~J. Luczak.
\newblock A new approach to the giant component problem.
\newblock {\em Random Structures \& Algorithms}, 34(2):197--216, 2009.

\bibitem{KolKom15}
I.~Kolossv{\'a}ry and J.~Komj{\'a}thy.
\newblock First passage percolation on inhomogeneous random graphs.
\newblock {\em Advances in Applied Probability}, 47(2):589?610, 2015.

\bibitem{Kom16}
J.~Komj{\'a}thy.
\newblock Explosive crump-mode-jagers branching processes.
\newblock {\em arXiv preprint arXiv:1602.01657}, 2016.

\bibitem{Milg67}
S.~Milgram.
\newblock The small world problem.
\newblock {\em Psychology Today}, May:60--67, 1967.

\bibitem{MonSol02}
J.~M. Montoya and R.~V. Sol{\'e}.
\newblock Small world patterns in food webs.
\newblock {\em Journal of theoretical biology}, 214(3):405--412, 2002.

\bibitem{Newm01}
M.~E.~J. Newman.
\newblock The structure of scientific collaboration networks.
\newblock {\em Proceedings of the National Academy of Sciences},
  98(2):404--409, 2001.

\bibitem{Red98}
S.~Redner.
\newblock How popular is your paper? an empirical study of the citation
  distribution.
\newblock {\em The European Physical Journal B-Condensed Matter and Complex
  Systems}, 4(2):131--134, 1998.

\bibitem{Reittu:2004}
H.~Reittu and I.~Norros.
\newblock On the power-law random graph model of massive data networks.
\newblock {\em Perform. Eval.}, 55(1-2):3--23, Jan. 2004.

\bibitem{WatStr98}
D.~J. Watts and S.~H. Strogatz.
\newblock Collective dynamics of `small-world' networks.
\newblock {\em Nature}, {\bf 393}:440--442, 1998.

\end{thebibliography}

\end{document}